\documentclass[11pt,leqno]{article}
\usepackage[T1]{fontenc}
\usepackage{authblk}
\usepackage[margin=1.25in]{geometry}
\usepackage{amsmath,amssymb,amsthm,mathtools,booktabs,longtable,array}
\usepackage{dsfont,cite,microtype,hyperref}
\hypersetup{colorlinks,breaklinks,linkcolor=blue,citecolor=blue,urlcolor=blue,
 pdftitle={On prediction from expert advice with more than five experts},
 pdfauthor={Jeff Calder and Nadejda Drenska}}

\newcommand{\R}{\mathbb{R}}

\renewcommand{\v}{\mathbf{v}}
\newcommand{\one}{\mathds{1}}
\newcommand{\E}{{\mathbb E}}

\renewcommand{\phi}{\varphi}
\newcommand{\eps}{\varepsilon}
\renewcommand{\P}{\mathbb{P}}
\DeclareMathOperator{\csch}{csch}

\newcommand{\dd}{\,dt}

\DeclareMathOperator{\tr}{tr}

\newtheorem{theorem}{Theorem}
\newtheorem{proposition}[theorem]{Proposition}
\newtheorem{lemma}[theorem]{Lemma}
\newtheorem{corollary}[theorem]{Corollary}
\theoremstyle{definition}
\newtheorem{remark}[theorem]{Remark}

\numberwithin{equation}{section}
\numberwithin{theorem}{section}

\title{On prediction from expert advice with more than five experts\thanks{{\bf Funding:} JC acknowledges funding from an Albert and Dorothy Marden Professorship, a Simons Fellowship, and National Science Foundation Grant DMS:2436333. ND acknowledges funding from National Science Foundation Grant DMS:2407839. {\bf Source Code:} The verification scripts referred to in the text are available at \url{https://github.com/jwcalder/Beyond-Five-Experts}. A Lean~4 formalization of the main results is archived at \url{https://doi.org/10.5281/zenodo.23136818}. The proof development and the preparation of this paper were assisted by ChatGPT and Claude. The authors take full responsibility for the contents of the paper.}}

\author[1]{Jeff Calder}
\affil[1]{School of Mathematics, University of Minnesota}
\author[2]{Nadejda Drenska}
\affil[2]{Department of Mathematics, Louisiana State University}

\begin{document}
\maketitle
\begin{abstract}
We prove that no single rank ordered adversary strategy is globally optimal for the prediction with expert advice problem with six or more experts, in both the geometric stopping and finite time horizon settings. The proof is based on establishing a leading order correction when one expert moves far ahead of the others. This allows us to connect optimal strategies between $n$ and $j<n$ experts and utilize recent results on the exact optimality set for the five expert problem. 
\end{abstract}
\tableofcontents

\section{Introduction}\label{sec:introduction}

The prediction with expert advice problem, dating back to Hannan~\cite{hannan1957approximation} and Cover~\cite{cover1967behavior}, involves a player attempting to synthesize advice from multiple experts, in order to make a prediction about future events. The player's performance is measured by their regret with respect to the best performing expert, and much of the analysis is in the \emph{worst case setting}, where we assume the experts are controlled by an adversary whose goal is to maximize the players regret (while the goal of the player is to minimize). Heuristic strategies, like the multiplicative weights algorithm~\cite{littlestone1994weighted,vovk1990aggregating,freund1997decision,cesa1997use,haussler1995tight} have long given good performance, but the question of optimal strategies has only been addressed more recently. 

The games involve repeated play between the player and the adversary for either a fixed number of steps known to both players (i.e., the \emph{finite-horizon} problem) or until a random stopping time that is geometrically distributed (i.e., the \emph{geometric stopping} problem). Gravin, Peres and Sivan~\cite{gravin2016towards} solved the geometric stopping game for three experts and showed that the adversary's optimal strategy is the so-called COMB strategy, which ranks the experts by their current gain and advances either the experts of odd rank or those of even rank, each with probability one half. They conjectured that COMB was asymptotically optimal for every number of experts. 

Drenska and Kohn~\cite{drenska2017pde,drenska2020prediction} developed a useful partial differential equations (PDE) framework for understanding the problem, which built upon previous work interpreting PDEs as continuum limits of two-player games~\cite{kohn2006deterministic,kohn2010deterministic,peres2009tug,peres2008tug}. Drenska and Kohn showed that as the number of turns in the game tends to infinity, the rescaled value function converges to the viscosity solution~\cite{crandall1992user} of a degenerate elliptic PDE in the geometric stopping setting, and of a degenerate parabolic PDE in the finite horizon setting~\cite{drenska2020prediction}. In the normalization of this paper, the geometric stopping PDE reads
 \begin{equation}\label{eq:pde_intro}
 u_n-\frac12\max_{\v\in\{0,1\}^n}\v^TD^2u_n\,\v=\max_i x_i \quad\text{on }\R^n,
 \end{equation}
where $x_i$ is the rescaled regret with respect to expert $i$. The solution of \eqref{eq:pde_intro} encodes the asymptotically optimal strategies for both players---in particular, the optimal strategy for the player is the probability distribution $\nabla u$ over the experts to follow, while the optimal strategy for the adversary is determined by the $\v$ that attains the maximum in \eqref{eq:pde_intro} at each $x$. In particular, the optimal adversary strategy is split the experts into two groups based on the positions of the ones and zeros in the optimal $\v$, and to advance one group or the other, each with probability one half. The approach has been extended to other related problems as well, such as history-dependent experts~\cite{drenska2023pde,calder2021asymptotically,calder2023online}. Since $u_n$ is invariant to permutations of coordinates, it suffices to consider the ordered sector $x_1\geq\cdots\geq x_n$, in which case the COMB strategy is written as $1010\cdots$. A direction $\v$ is globally optimal if it attains the maximum in \eqref{eq:pde_intro} at every point $x$ in the ordered sector; when $u_n$ is not known to be twice differentiable, this is understood in the weak sense of Section~\ref{sec:main_results}. Using the Drenska-Kohn framework, COMB has been shown to be globally optimal for $n\leq4$~\cite{abbasi2017near,gravin2016towards,drenska2020prediction,bayraktar2020asymptotic,bayraktar2020finite}, while Kobzar, Kohn and Wang~\cite{kobzar2020newprediction} showed that, in the finite-horizon problem and at leading order as the horizon tends to infinity, COMB is at least as powerful as a fully randomized adversary for any number of experts. There is extensive related work treatng regimes in which random play is optimal~\cite{abernethy2008random}, lower bounds for multiplicative weights~\cite{gravin2017tight}, an unknown horizon~\cite{luo2014towards}, potential-based bounds for general $n$~\cite{kobzar2020newprediction,kobzar2020newgeometric}, and limited or malicious adversaries~\cite{bayraktar2021prediction,bayraktar2021malicious}.

The first evidence that the COMB conjecture may be false in general appeared in \cite{chase2019experimental,calder2026numerical}. Based on computed numerical solutions of \eqref{eq:pde_intro} up to $n=10$ experts, Calder, Drenska, and Mosaphir \cite{calder2026numerical} made three conjectures: (i)~COMB is globally optimal only for $n\leq4$; (ii)~for $n=5$, the direction $10100$ is the only globally optimal direction; and (iii)~for $n\geq6$, no direction is globally optimal. In addition, it was left as an open problem in \cite{calder2026numerical} to derive an explicit formula for the solution based on the conjectured optimality of the direction $10100$. Calder and Drenska showed that the direction $10100$ is globally optimal in both the geometric \cite{calder2026explicit} and finite horizon \cite{calder2026finite} settings by finding explicit solutions. Independently Bayraktar, Ekren, and Kolliopoulos \cite{bayraktar2026prediction} arrived at an equivalent solution formula for the geometric problem for $n=5$ experts. 

In this paper, we prove conjecture (iii) in both the geometric stopping and finite horizon settings, and we complete the proof of conjecture (ii) by showing that, up to complementation, $10100$ is the only globally optimal direction for five experts. The proof of conjecture (iii) has two parts. We first identify the leading-order behavior of $u_n$ when one expert is far ahead of the others, and deduce that a globally optimal direction selects at most one expert besides the leader. Since the first five bits of a globally optimal direction are optimal for five experts, the five-expert uniqueness above forces these bits to be $10100$. This already selects one expert besides the leader, so all later bits vanish and only the candidate $1010\cdots0$ remains. We exclude this candidate using trace equations on the sector where the bottom $n-3$ experts are tied, which show that global optimality would force the second-order growth of $u_n$ along the ray $(s,s,s,0,\dots,0)$ to exceed what \eqref{eq:pde_intro} allows when $n\geq6$. A Laplace transform in the horizon then transfers the result to the finite-horizon problem. The techniques in the paper also produce other intermediate results about optimal strategies when one expert is far ahead, and give some new bounds on the value function $u_n(0)$ at the origin, which are described in Section \ref{sec:main_results}.

Together with the companion papers~\cite{calder2026explicit,calder2026finite}, these results resolve all three conjectures of~\cite{calder2026numerical}, as well as the open problem posed there. Conjecture (i) holds: COMB is globally optimal for $n\leq4$~\cite{gravin2016towards,drenska2020prediction,bayraktar2020asymptotic}, it is not globally optimal for $n=5$~\cite{calder2026explicit,calder2026finite}, and Theorem~\ref{thm:nonexistence} excludes it, along with every other direction constant in rank order, for $n\geq6$. Conjecture (ii) holds: $10100$ is globally optimal for five experts~\cite{calder2026explicit,calder2026finite}, and Theorem~\ref{thm:five} shows that, up to complementation, no other direction is. Conjecture (iii) is Theorem~\ref{thm:nonexistence}, and Theorem~\ref{thm:finite} extends it to the finite-horizon problem. The explicit five-expert solutions of~\cite{calder2026explicit,calder2026finite} answer the open problem. In particular, COMB is not globally optimal for any $n\geq5$. Our results do not identify the optimal adversarial strategy for $n\geq6$, which must vary with the state within the ordered sector, and they do not settle whether $u_n$ is $C^2$ for $n\geq6$.

No result of this paper depends on the numerical computations of~\cite{calder2026numerical}, which serve only as motivation. We use the five-expert results of~\cite{calder2026explicit} in the proof of Theorem~\ref{thm:five}, and the trace equations of~\cite[\S3.2]{calder2026explicit} in Section~\ref{sec:last}, where the propagator formula of~\cite[Lemma~3.3]{calder2026explicit} is verified directly; the verification in~\cite{calder2026explicit} that $10100$ attains the maximum in the Hamiltonian is computer assisted. The main results are also formalized in the Lean~4 proof assistant~\cite{calder2026lean}; the formalization takes as hypotheses the properties of the geometric and finite-horizon value functions that we recall from the literature or derive from their control representations.

The rest of the paper is organized as follows. Section~\ref{sec:main_results} states the main results. Section~\ref{sec:basic} collects preliminary facts about $u_n$, including the linear supersolution inequalities, the $C^{1,1}$ estimate, and prefix inheritance, and then uses prefix inheritance to prove that, up to complementation, $10100$ is the only globally optimal direction for five experts. Section~\ref{sec:profile} establishes the one-leader asymptotics and excludes directions with three or more ones. Section~\ref{sec:last} excludes the last candidate $1010\cdots0$, and Section~\ref{sec:finite} treats the finite-horizon problem.
\section{Main results}\label{sec:main_results}

Let $u_n$ be the asymptotic value function for the $n$-expert geometric stopping time problem, which is the viscosity solution of the PDE
\begin{equation}\label{eq:pde}
 u_n-\frac12\max_{\v\in\{0,1\}^n}\v^TD^2u_n\,\v=\phi_n,
 \qquad \phi_n(x)=\max_i x_i.
\end{equation}
We recall that $u_n$ is invariant to permutations of the coordinates, is translation equivariant in the ones direction, i.e., $u_n(x + s\one) = u_n(x) + s$, and $u_n-\phi_n$ is bounded \cite{calder2026numerical}. For $n\leq 5$ it is known that $u_n\in C^2(\R^n)$ \cite{bayraktar2020asymptotic,calder2026explicit}, while for $n\geq 6$ classical regularity is an open problem, and $u_n$ must be interpreted in the weaker viscosity sense \cite{crandall1992user}. In fact, we show in Section~\ref{sec:basic} that $u_n\in C^{1,1}(\R^n)$ for all $n$, which is used throughout the paper, and Lemma~\ref{lem:regularity} below gives an explicit bound on its Hessian. 

Writing $D_{\v}=\sum_iv_i\partial_{x_i}$ for the derivative in the direction $\v$, we have $\v^TD^2u_n\,\v=D_{\v}^2u_n$, and we call $\max_{\v}D_{\v}^2u_n$ the Hamiltonian. Directions are written in decreasing rank order and complementation identifies equivalent directions. Indeed, when $u$ is smooth, the translation equivariance implies $\nabla u_n \cdot \one = 1$ and so $D^2u_n\one = 0$, from which it follows that the directions $\v$ and $\one-\v$ yield the same quadratic form. Thus, we normalize the first bit of any $\v$ we work with to be one. Global optimality of $\v$ entails the solution to \eqref{eq:pde} $u_n$ solving the linear equation
\begin{equation}\label{eq:fixed}
 u_n-\frac12D_{\v}^2u_n=x_1
 \quad\text{on }\{x_1>x_2>\cdots>x_n\},
\end{equation}
as well, where the equality is understood in the viscosity sense. The normalization of the first bit is legitimate also for this viscosity formulation, by the explicit form of \eqref{eq:fixed} on segments (see the discussion after Lemma~\ref{lem:lines}). For a constant rank direction $\v$, global optimality implies that $\v$ attains the Hamiltonian maximum at every point of the strictly ordered sector, and hence \eqref{eq:fixed} holds there. All subsequent nonexistence results are established by disproving this necessary condition. 

Throughout the paper we set $\kappa=\sqrt2$ and $c=1/(2\kappa)$, the constants in the two-expert solution $u_2(x)=\phi_2(x)+c\,e^{-\kappa|x_1-x_2|}$~\cite[(3.2)]{calder2026explicit}. Our main theorems are based on a new result establishing the leading order behavior of the value function $u_n$ when one expert is far ahead of the others.
\begin{theorem}[Universal one-leader profile]\label{thm:profile}
For $m\geq1$, let $E_1,\ldots,E_m$ be independent exponential random variables of rate $2\kappa$, and define $h_m:\R^m\to\R$ by
\begin{equation}\label{eq:profile}
 h_m(z)=\frac1{4\kappa}\E\exp\!\left(\kappa
                 \max_{1\leq i\leq m}(z_i+E_i)\right).
\end{equation}
For $n\geq2$, write $x=(s,z)$ with $s\in\R$ and $z\in\R^{n-1}$. Then
\[
 u_n(s,z)-s\sim e^{-\kappa s}h_{n-1}(z)
 \quad\text{as }s\to\infty,
\]
locally uniformly in $z\in\R^{n-1}$. Moreover, for each fixed $z$, the function $s\mapsto e^{\kappa s}(u_n(s,z)-s)$ is nondecreasing on $(\max_iz_i,\infty)$.
\end{theorem}
\begin{remark}[Notation]\label{rem:asymptotics}
In Theorem \ref{thm:profile}, and the rest of this paper, we use the asymptotic notation $f\sim g$ as $s\to s_0$ to mean that $\lim_{s\to s_0}f(s)/g(s)=1$. When $f$ and $g$ also depend on other variables, we say that $f\sim g$ holds locally uniformly in those variables if $f/g\to1$ uniformly on compact sets. 
\end{remark}
The profile $h_m$ is an elementary finite sum of exponentials (Lemma~\ref{lem:h_m}); in particular $h_m(0)=(2m)!!/(4\kappa(2m-1)!!)$, so for six experts 
\[u_6(s,0,\dots,0)-s\sim\tfrac{64}{63\sqrt2}e^{-\sqrt2s} \ \ \text{as}  \ \ s\to\infty.\]

The definition of the profile also allows us to explicitly state our $C^{1,1}$ regularity result. 
\begin{lemma}[Regularity]\label{lem:regularity}
For every $n\geq2$, $u_n\in C^{1,1}(\R^n)$, and
\[
 0\preceq D^2u_n(x)\preceq8\bigl(u_n(x)-\phi_n(x)\bigr)I\preceq8h_{n-1}(0)\,e^{-\sqrt2(x_{(1)}-x_{(2)})}I
\]
for almost every $x\in\R^n$, where $x_{(1)}\geq\cdots\geq x_{(n)}$ is the decreasing rearrangement of $x$. In particular, $\nabla u_n$ is Lipschitz on $\R^n$ with constant $8h_{n-1}(0)$.
\end{lemma}
The player's asymptotically optimal strategy $\nabla u_n$ follows expert $i$ with probability $\partial_{x_i}u_n$. Lemma~\ref{lem:regularity} shows that this strategy is Lipschitz, and that its local Lipschitz constant decays exponentially in the gap between the two leading experts. The lemma is proved at the end of Section~\ref{sec:barrier}.

The convergence in Theorem~\ref{thm:profile} also holds for first derivatives, which describe the player's asymptotically optimal strategy.
\begin{corollary}[The player's strategy when one expert leads]\label{cor:player}
In the setting of Theorem~\ref{thm:profile},
\begin{equation}\label{eq:player}
 \partial_{z_i}u_n(s,z)\sim e^{-\kappa s}\,\partial_{z_i}h_{n-1}(z)\quad(1\leq i\leq n-1),\qquad
 1-\partial_su_n(s,z)\sim\kappa e^{-\kappa s}h_{n-1}(z)
\end{equation}
as $s\to\infty$, locally uniformly in $z\in\R^{n-1}$. In particular, the player follows the trailing expert with coordinate $z_i$ with probability $\partial_{z_i}u_n(s,z)\sim e^{-\sqrt2s}\,\partial_{z_i}h_{n-1}(z)$. When the trailing experts are tied, each is followed with probability $\partial_{z_i}u_n(s,0)\sim\kappa h_{n-1}(0)e^{-\sqrt2s}/(n-1)$, which is $\tfrac{64}{315}e^{-\sqrt2s}$ for six experts.
\end{corollary}
Kobzar, Kohn and Wang~\cite{kobzar2020newprediction,kobzar2020newgeometric} construct explicit player strategies with regret bounds for every $n$; Corollary~\ref{cor:player} identifies the optimal strategy itself when one expert leads.

We now turn to globally optimal directions. Two reductions limit the candidates. The first is prefix inheritance: if $\v\in\{0,1\}^n$ satisfies \eqref{eq:fixed}, then for every $j<n$ its first $j$ bits satisfy \eqref{eq:fixed} with $n$ replaced by $j$ (Lemma~\ref{lem:prefix}). Indeed, when the bottom $n-j$ experts trail far behind, $u_n$ is close to $u_j$ evaluated at the top $j$ experts (Lemma~\ref{lem:removal}), and \eqref{eq:fixed} passes to the limit. The second reduction is our first main result.
\begin{theorem}[Uniqueness for five experts]\label{thm:five}
For $n=5$, the only direction $\v$ with first bit one that satisfies \eqref{eq:fixed} is $\v=10100$. In particular, up to complementation, no rank direction other than $10100$ is globally optimal for five experts.
\end{theorem}
Calder and Drenska~\cite[Theorem~2.1]{calder2026explicit} showed that $10100$ is globally optimal, so Theorem~\ref{thm:five} completes the proof of conjecture (ii) of~\cite{calder2026numerical}. The proof, in Section~\ref{sec:inheritance}, combines prefix inheritance with the results of~\cite{calder2026explicit}. Together, prefix inheritance and Theorem~\ref{thm:five} show that for every $n\geq5$ a direction satisfying \eqref{eq:fixed} begins with $10100$. Our next result uses Theorem~\ref{thm:profile} to show that any direction satisfying \eqref{eq:fixed} selects at most one expert besides the leader, which leaves a single candidate.
\begin{theorem}[At most two ones]\label{thm:exclude}
A direction $\v$ with first bit one that satisfies \eqref{eq:fixed} has at most two ones. Consequently, for every $n\geq5$, the only direction with first bit one that can satisfy \eqref{eq:fixed} is
\[
 (1,0,1,0,\ldots,0).
\]
\end{theorem}
For six experts, this leaves only $101000$, which needs a separate result to exclude.

\begin{theorem}[Nonexistence for six or more experts]\label{thm:nonexistence}
For every $n\geq6$, no binary direction $\v$ satisfies \eqref{eq:fixed} throughout the strictly ordered sector. In particular, the geometric prediction problem has no globally optimal adversarial direction that is constant in rank order.
\end{theorem}
By Theorem~\ref{thm:exclude}, it remains to exclude $1010\cdots0$, which Section~\ref{sec:last} does. The argument uses the established convexity and $C^{1,1}$ bound of the true Bellman value $u_n$ under the hypothesis of global optimality; it makes no regularity assertion about arbitrary fixed-policy values.

\begin{remark}[Randomized rank strategies]\label{rem:mixture}
In the game of Drenska and Kohn, the adversary chooses in each round an arbitrary probability distribution on the outcomes $\{0,1\}^n$~\cite[Section~2.1]{drenska2020prediction}. In particular, it may randomize over rank directions: in every round it draws a direction $\v$ with a fixed probability $p_{\v}$, reads it in the current rank order, and advances either the experts selected by $\v$ or those selected by $\one-\v$, each with probability one half. Drenska and Kohn observed that randomizing does not increase the Hamiltonian, since a weighted average of $D_{\v}^2u_n$ never exceeds the largest of them~\cite[Section~3.3]{drenska2020prediction}. The equality case of this observation shows that, for $n\geq6$, Theorem~\ref{thm:nonexistence} also rules out these strategies. Indeed, the linear equation of such a mixture replaces $D_{\v}^2$ in \eqref{eq:fixed} by $\sum_{\v}p_{\v}D_{\v}^2$. Where the Hessian of $u_n$ exists, the residual of the mixture is the weighted average $\sum_{\v}p_{\v}R_{\v}$ of the residuals $R_{\v}=u_n-\tfrac12D_{\v}^2u_n-\phi_n$, which are nonnegative by \eqref{eq:linearineq}, so it vanishes only if $R_{\v}=0$ for every $\v$ with $p_{\v}>0$. Since $u_n\in C^{1,1}$ by Lemma~\ref{lem:regularity}, its Hessian exists almost everywhere, and viscosity and distributional solutions of these linear equations coincide~\cite{ishii1995equivalence}. Hence every direction in the support of a mixture whose equation holds throughout the sector satisfies \eqref{eq:fixed}. Like Theorem~\ref{thm:nonexistence}, this concerns strategies that are constant in rank order and says nothing about state-dependent mixtures.
\end{remark}

We also note that the barrier functions used in our proofs lead to explicit bounds on the value function, some of which are sharper than existing results. 
\begin{remark}[Explicit bounds on $u_n$]\label{rem:origin}
Comparison with the global supersolution of Proposition~\ref{prop:barrier} gives
\[
 u_n(x)-\phi_n(x)\leq h_{n-1}\bigl(x_{(2)}-x_{(1)},\dots,x_{(n)}-x_{(1)}\bigr)\qquad\text{for all }x\in\R^n,
\]
with equality for $n=2,3$, where that supersolution is $u_n$ itself. By Theorem~\ref{thm:profile} the bound is asymptotically sharp when one expert leads: with $x=(s,z)$, both sides are $\sim e^{-\kappa s}h_{n-1}(z)$ as $s\to\infty$. At the origin it gives $u_n(0)\leq h_{n-1}(0)$, but $h_{n-1}(0)$ describes the tail of $u_n$ rather than its value when all experts are tied, and we now turn to sharper bounds on $u_n(0)$.

The value $u_n(0)$ is known explicitly only for $n\leq5$, where it equals $\sqrt2/4$, $\sqrt2/3$, $\pi/(4\sqrt2)$ and $45\pi^2/(512\sqrt2)$ for $n=2,3,4,5$~\cite{calder2026explicit}. For larger $n$, the previously known bounds are those of Kobzar, Kohn and Wang~\cite{kobzar2020newgeometric}, who obtain bounds at every state from explicit sub- and supersolutions; in our normalization their limiting value is $2u_n(x/2)$. At the origin, their heat-based adversary and their exponential-weights, heat-based and max-potential players give
\[
 \frac{\sqrt\pi}4\sqrt{2c_n}\,\E\max_{1\leq i\leq n}Z_i\;\leq\;u_n(0)\;\leq\;\min\Bigl\{\sqrt{\frac{\log n}2},\ \sqrt{\frac\pi8}\,\E\max_{1\leq i\leq n}Z_i,\ M_n\Bigr\},
\]
where $Z_1,\dots,Z_n$ are independent standard normal random variables, $2c_n=n/(n-1)$ and $M_n=\sqrt{(n-1)/8}$ for even $n$, and $2c_n=(n+1)/n$ and $M_n=\tfrac{n-1}n\sqrt{(n+1)/8}$ for odd $n$. The lower bound exceeds $u_5(0)$ for every $n\geq6$, and they showed that it matches the exponential-weights upper bound up to the factor $\sqrt\pi/2$.

Our results give two new upper bounds at the origin (Lemma~\ref{lem:origin} in Appendix~\ref{app:origin}). First, convexity sharpens the barrier: for $n\geq4$,
\[
 u_n(0)\leq B_n:=\frac{n-1}{n\kappa}\Bigl(1+\log\frac{n\kappa h_{n-1}(0)}{n-1}\Bigr)<h_{n-1}(0),
 \qquad h_{n-1}(0)=\frac{(2n-2)!!}{4\kappa(2n-3)!!}.
\]
Second, comparison with the finite-horizon problem (Remark~\ref{rem:finitegeometric}) improves the exponential-weights bound by the factor $\sqrt\pi/2$, which closes the gap above:
\[
 u_n(0)\leq\sqrt{\frac{\pi\log n}{8}}\qquad\text{and}\qquad u_n(0)\sim\sqrt{\frac{\pi\log n}{8}}\quad\text{as }n\to\infty.
\]
Each improves on all three previous upper bounds over a range of $n$: $B_n$ for $6\leq n\leq23$ and $\sqrt{\pi\log n/8}$ for every $n\geq9$ (Appendix~\ref{app:origin}, where finitely many numerical inequalities are certified by computer). The smaller of the two, and hence the smallest of the upper bounds compared here, is $B_n$ for $6\leq n\leq13$ and $\sqrt{\pi\log n/8}$ for $n\geq14$. Combining the monotonicity $u_n(0)\geq u_{n-1}(0)$ of Lemma~\ref{lem:removal} with the exact values gives $u_n(0)\geq u_5(0)$ for $n\geq6$, which is weaker than the heat-based lower bound. For $n\leq5$ the exact values agree with $h_{n-1}(0)$ when $n=2,3$ and are smaller than every upper bound above when $n=4,5$. %
\end{remark}

Finally, the geometric statement implies its finite-horizon analogue. Write $U_n(\tau,x)$ for the value with time to maturity $\tau$,
\begin{equation}\label{eq:parabolic}
 \partial_\tau U_n=\frac12\max_{\v\in\{0,1\}^n}D_{\v}^2U_n
 \quad\text{on }(0,\infty)\times\R^n,
 \qquad U_n(0,\cdot)=\phi_n .
\end{equation}

\begin{theorem}[Finite-horizon nonexistence]\label{thm:finite}
For every $n\geq6$ there is no binary direction $\v$ with
\begin{equation}\label{eq:parfixed}
 \partial_\tau U_n=\frac12D_{\v}^2U_n
 \quad\text{on }(0,\infty)\times\{x_1>\cdots>x_n\}.
\end{equation}
The same holds with $(0,\infty)$ replaced by any nonempty open time interval, and no rank direction attains the Hamiltonian maximum throughout the spatial sector at even a single positive time to maturity, in the sense of Remark~\ref{rem:selfsimilar}. Directions allowed to depend on the remaining time do not help.
\end{theorem}
Section~\ref{sec:finite} proves Theorem~\ref{thm:finite} from Theorem~\ref{thm:nonexistence} by a Laplace transform in $\tau$. The transfer is applied to the linear equations attached to a hypothetical common optimizer, never to the nonlinear Hamiltonian, so no claim is made that a Laplace transform commutes with the Hamiltonian.

\begin{remark}[Transfer from finite to geometric horizon]\label{rem:transfer}
The transfer behind Theorem~\ref{thm:finite} holds for every $n$ (Lemma~\ref{lem:transferparabolic}): a rank direction that attains the maximum in \eqref{eq:parabolic} throughout the sector at every time to maturity also attains the maximum in \eqref{eq:pde} throughout the sector. The result holds without a classical $C^{1,2}$ regularity assumption on $U_n$. Bayraktar, Ekren and Zhang~\cite[Proposition~3.1]{bayraktar2020finite} proved this transfer for COMB under the additional assumptions that $U_n$ is $C^{1,2}$ and that its spatial Hessian is integrable against $e^{-\tau}\,d\tau$ (see Remark~\ref{rem:bez}). Kobzar, Kohn and Wang~\cite{kobzar2020newgeometric} use the same transform in time to pass explicit bounds from the finite horizon to the geometric problem. Combined with Theorem~\ref{thm:five} and~\cite{calder2026finite}, the transfer shows that, up to complementation, $10100$ is the only rank direction that is optimal throughout the sector at every time to maturity for five experts. By self-similarity, the set of rank directions optimal throughout the sector does not depend on the time to maturity (Remark~\ref{rem:selfsimilar}).
\end{remark}

\begin{remark}[Finite versus geometric horizon]\label{rem:finitegeometric}
For every $n$ and $x\in\R^n$,
\begin{equation}\label{eq:finitegeometric}
 u_n(x)\leq\int_0^\infty e^{-\tau}U_n(\tau,x)\,d\tau.
\end{equation}
Indeed, $u_n(x)=\sup_\alpha\E\int_0^\infty e^{-t}\phi_n(X^\alpha_t)\,dt$ is the value of a discounted control problem, where $X^\alpha_t=x+\int_0^t\alpha_s\,dW_s$, $W$ is a scalar Brownian motion, and the supremum is over progressively measurable controls $\alpha$ with values in $\{0,1\}^n$. The finite-horizon value is $U_n(t,x)=\sup_\alpha\E\phi_n(X^\alpha_t)$ (see \eqref{eq:parcontrol}), so $\E\phi_n(X^\alpha_t)\leq U_n(t,x)$ for every $t$. Thus a control for the geometric problem must serve every horizon at once, while $U_n(\tau,x)$ optimizes each horizon separately. Lemma~\ref{lem:transferparabolic} gives equality in \eqref{eq:finitegeometric} whenever some rank direction satisfies \eqref{eq:parfixed}, so equality holds for $n\leq5$ by~\cite{bayraktar2020finite} and~\cite{calder2026finite}. By the self-similarity \eqref{eq:parscaling}, $U_n(\tau,0)=\sqrt\tau\,U_n(1,0)$, so at the origin \eqref{eq:finitegeometric} reads
\[
 U_n(1,0)\geq\frac{2}{\sqrt\pi}\,u_n(0),
\]
with equality for $n\leq5$, as Bayraktar, Ekren and Zhang~\cite[Corollary~3.1]{bayraktar2020finite} showed for $n=3,4$. Lower bounds on $u_n(0)$ therefore transfer to $U_n$, but upper bounds do not transfer this way. In the other direction, \eqref{eq:finitegeometric} turns upper bounds on $U_n$ into upper bounds on $u_n$. For each $\eta>0$, the potential $\eta^{-1}\log\sum_ie^{\eta x_i}+\eta\tau/8$ of the exponential-weights player with learning rate $\eta$, who follows expert $i$ with probability proportional to $e^{\eta x_i}$~\cite{kobzar2020newprediction}, is a supersolution of \eqref{eq:parabolic}, because its second derivative in every binary direction is at most $\eta/4$. Comparison with $\eta=\sqrt{8(\log n)/\tau}$ gives $U_n(\tau,0)\leq\sqrt{\tau(\log n)/2}$, and \eqref{eq:finitegeometric} then gives
\[
 u_n(0)\leq\Gamma\Bigl(\frac32\Bigr)\sqrt{\frac{\log n}{2}}=\sqrt{\frac{\pi\log n}{8}},
\]
the bound used in Remark~\ref{rem:origin}. The geometric exponential-weights bound $\sqrt{(\log n)/2}$ optimizes a single fixed learning rate. Here the learning rate is optimized separately for each deterministic horizon $\tau$, which gains the factor $\sqrt\pi/2$. Unlike the bounds of~\cite{kobzar2020newgeometric}, this argument does not by itself specify a player strategy for the geometric problem. Such a strategy is provided, to leading order as $n\to\infty$, by the recent algorithm of Cai et al.~\cite{cai2026prediction}. It needs no knowledge of the horizon, has regret at most $\bigl(1+O(\sqrt{\log\log n/\log n})\bigr)\sqrt{t\log n/2}$ at every time $t$, and therefore attains $\sqrt{\pi\log n/8}$ under geometric stopping, up to a factor $1+o(1)$. Since $U_n(1,0)\leq\sqrt{(\log n)/2}$ and $u_n(0)\sim\sqrt{\pi(\log n)/8}$, the inequality $U_n(1,0)\geq2u_n(0)/\sqrt\pi$ gives $U_n(1,0)\sim\sqrt{(\log n)/2}$, so the ratio $U_n(1,0)/u_n(0)$ tends to $2/\sqrt\pi$ as $n\to\infty$. Whether it equals $2/\sqrt\pi$ for any $n\geq6$ is open.
\end{remark}

The rest of the paper proves the results described in this section. Section~\ref{sec:basic} collects the facts about $u_n$ used throughout: the restriction of \eqref{eq:pde} and \eqref{eq:fixed} to lines, the linear supersolution inequalities, the $C^{1,1}$ bound, the removal estimate, and prefix inheritance. It then combines prefix inheritance with~\cite{calder2026explicit} to prove Theorem~\ref{thm:five}, which forces every candidate direction to begin with $10100$. Section~\ref{sec:profile} proves Theorem~\ref{thm:profile} by matching an explicit global supersolution with a lower bound obtained from a minimality property of the profile, and then deduces Lemma~\ref{lem:regularity}, Corollary~\ref{cor:player} and Theorem~\ref{thm:exclude}. Section~\ref{sec:last} excludes the remaining candidate $1010\cdots0$ through the boundary-trace equations, proving Theorem~\ref{thm:nonexistence}. Section~\ref{sec:finite} transfers the result to the finite horizon, proving Theorem~\ref{thm:finite}.
\section{Preliminaries}\label{sec:basic}

This section assembles the basic tools used in the rest of the paper, ending with the proof of Theorem~\ref{thm:five}. Convexity and the standard symmetries of the prediction value follow either from its discounted control representation (Remark~\ref{rem:finitegeometric}) or from the comparison arguments in~\cite{calder2026numerical}. In particular $u_n$ is convex, permutation invariant, and satisfies $u_n(x+a\one)=u_n(x)+a$, so that $\nabla u_n\cdot\one=1$ and $D^2u_n\one=0$ wherever these derivatives exist. Throughout this paper we make use of the theory of viscosity solutions of PDEs; we refer the reader to \cite{crandall1992user} for more details. 

\subsection{Linear inequalities and \texorpdfstring{$C^{1,1}$}{C1,1} regularity}

The regularity of $u_n$ used throughout the paper comes from linear inequalities, which we obtain by restricting \eqref{eq:pde} to lines. For a direction $\v\neq0$, an open set $\Omega\subseteq\R^n$ and $f\in C(\Omega)$, a continuous function $w$ on $\Omega$ is a viscosity supersolution of $w-\tfrac12D_{\v}^2w\geq f$ in $\Omega$ if $w(x_0)-\tfrac12D_{\v}^2\psi(x_0)\geq f(x_0)$ whenever $\psi$ is $C^2$ and $w-\psi$ has a local minimum at $x_0\in\Omega$. Subsolutions are defined with a local maximum and the reverse inequality, and a viscosity solution of $w-\tfrac12D_{\v}^2w=f$ is both. If $\psi$ touches $u_n$ from below at $x_0$, the supersolution half of \eqref{eq:pde} reads $u_n(x_0)-\tfrac12\max_{\v'\in\{0,1\}^n}D^2_{\v'}\psi(x_0)\geq\phi_n(x_0)$, and discarding all the competitors $\v'\ne\v$ shows that, for each fixed $\v\in\{0,1\}^n$, $u_n$ is a viscosity supersolution of $u_n-\tfrac12D_{\v}^2u_n\geq\phi_n$ on $\R^n$. The next lemma reduces such statements to one dimension, where they can be compared with explicit solutions.

\begin{lemma}\label{lem:lines}
Let $\v\neq0$, let $\Omega\subseteq\R^n$ be open, and let $f,w\in C(\Omega)$. Suppose that $x+t\v\in\Omega$ for $a\leq t\leq b$, and write $g(t)=w(x+t\v)$ and $\tilde f(t)=f(x+t\v)$. If $w$ is a viscosity supersolution of $w-\tfrac12D_{\v}^2w\geq f$ in $\Omega$, then the following hold.
\begin{enumerate}
\item[(i)] The function $g$ is a viscosity supersolution of $g-\tfrac12g''\geq\tilde f$ on $(a,b)$.
\item[(ii)] If $H\in C^2((a,b))\cap C([a,b])$ satisfies $H-\tfrac12H''\leq\tilde f$ on $(a,b)$, $H(a)\leq g(a)$ and $H(b)\leq g(b)$, then $H\leq g$ on $[a,b]$.
\item[(iii)] If $[t-h,t+h]\subseteq[a,b]$, then
\begin{equation}\label{eq:seconddiff}
 g(t+h)+g(t-h)-2g(t)\leq2h^2\max_{|s-t|\leq h}\bigl(g(s)-\tilde f(s)\bigr).
\end{equation}
\end{enumerate}
For subsolutions, (i) and (ii) hold with all inequalities reversed. Moreover, the following hold.
\begin{enumerate}
\item[(iv)] If $w$ is a viscosity solution of $w-\tfrac12D_{\v}^2w=f$ in $\Omega$ and $\tilde f(t)=\alpha+\beta t$ on $[a,b]$, then
\begin{equation}\label{eq:linesol}
 w(x+t\v)=\alpha+\beta t+Ae^{\kappa t}+Be^{-\kappa t}\qquad(a\leq t\leq b)
\end{equation}
for some constants $A$ and $B$.
\item[(v)] If $f$ is affine on every line in the direction $\v$, and every point $y\in\Omega$ is the midpoint of a segment in the direction $\v$ on which $w(y+t\v)=f(y+t\v)+Ae^{\kappa t}+Be^{-\kappa t}$ for some constants $A$ and $B$, then $w$ is a viscosity solution of $w-\tfrac12D_{\v}^2w=f$ in $\Omega$.
\end{enumerate}
\end{lemma}

\begin{proof}
We first prove (i) by penalizing the directions transverse to $\v$. Let $\psi\in C^2(\R)$ and let $g-\psi$ have a local minimum at $t_0\in(a,b)$; replacing $\psi$ by $\psi-(t-t_0)^4$, we may assume that the minimum is strict on $[t_0-r,t_0+r]\subset(a,b)$. Write $y=x+\tau(y)\v+\pi(y)$ with $\tau(y)=(y-x)\cdot\v/|\v|^2$ and $\pi(y)\cdot\v=0$, and for $\eps>0$ let $\psi_\eps(y)=\psi(\tau(y))-|\pi(y)|^2/\eps$. Let $y_\eps$ minimize $w-\psi_\eps$ over the compact tube $T=\{y:|\tau(y)-t_0|\leq r,\,|\pi(y)|\leq r\}$, which lies in $\Omega$ for $r$ small. Comparing with the value at $x+t_0\v$ shows that $|\pi(y_\eps)|^2/\eps$ is bounded and that every limit point $\tau_*$ of $\tau(y_\eps)$ satisfies $(g-\psi)(\tau_*)\leq(g-\psi)(t_0)$, so $\tau_*=t_0$. Hence $y_\eps\to x+t_0\v$, and for small $\eps$ the point $y_\eps$ is an interior local minimum of $w-\psi_\eps$. Since $\tau(y+s\v)=\tau(y)+s$ and $\pi(y+s\v)=\pi(y)$, we have $D_{\v}^2\psi_\eps(y_\eps)=\psi''(\tau(y_\eps))$, and the supersolution property gives $w(y_\eps)-\tfrac12\psi''(\tau(y_\eps))\geq f(y_\eps)$. Letting $\eps\downarrow0$ gives $g(t_0)-\tfrac12\psi''(t_0)\geq\tilde f(t_0)$.

For (ii), suppose that $g-H$ has a negative minimum $m$ on $[a,b]$. It is attained at some $t_*\in(a,b)$, and $H+m$ touches $g$ from below there, so $g(t_*)-\tfrac12H''(t_*)\geq\tilde f(t_*)\geq H(t_*)-\tfrac12H''(t_*)$, that is, $g(t_*)\geq H(t_*)$, which contradicts $m<0$.

For (iii), let $K=\max_{|s-t|\leq h}(g-\tilde f)$. If $\psi$ touches $g$ from below at a point $s$ with $|s-t|<h$, then (i) gives $\psi''(s)\leq2(g-\tilde f)(s)\leq2K$. For $\delta>0$ let $P$ be the quadratic polynomial with $P''=2K+2\delta$ and $P(t\pm h)=g(t\pm h)$. If $g-P$ had a negative minimum at a point $s$ with $|s-t|<h$, then $P$ plus a constant would touch $g$ from below at $s$ with second derivative $2K+2\delta>2K$, which is impossible. Hence $g\geq P$ on $[t-h,t+h]$, and at the midpoint this reads $g(t)\geq\tfrac12\bigl(g(t+h)+g(t-h)\bigr)-(K+\delta)h^2$. Letting $\delta\downarrow0$ gives \eqref{eq:seconddiff}.

The proofs for subsolutions are the same. For (iv), let $H(t)=\alpha+\beta t+Ae^{\kappa t}+Be^{-\kappa t}$ with $A$ and $B$ chosen so that $H(a)=g(a)$ and $H(b)=g(b)$, which is possible because $a\neq b$. Since $\kappa^2=2$, $H-\tfrac12H''=\alpha+\beta t$, and (ii) together with its version for subsolutions gives $H\leq g\leq H$. For (v), let $\psi$ touch $w$ from below at $y\in\Omega$, and let $g(t)=w(y+t\v)$ have the stated form near $t=0$. Then $g(t)-\psi(y+t\v)$ has a local minimum at $t=0$, so $D_{\v}^2\psi(y)\leq g''(0)=2\bigl(g(0)-f(y)\bigr)=2\bigl(w(y)-f(y)\bigr)$, because $f(y+t\v)$ is affine in $t$; touching from above gives the reverse inequality.
\end{proof}

Lemma~\ref{lem:lines} also shows that the normalization of the first bit in Section~\ref{sec:main_results} is legitimate for the viscosity formulation of \eqref{eq:fixed}. Let $\v\notin\{0,\one\}$. By (iv) and (v), \eqref{eq:fixed} holds for $\v$ if and only if $u_n$ has the form \eqref{eq:linesol}, with $\alpha=x_1$ and $\beta=v_1$, on every segment of the sector in the direction $\v$. Translation equivariance gives $u_n(x+t(\one-\v))=u_n(x-t\v)+t$, and the sector is invariant under $x\mapsto x+t\one$, so these representations for $\v$ turn into those for $\one-\v$ with $A$ and $B$ exchanged. Hence \eqref{eq:fixed} holds for $\v$ if and only if it holds for $\one-\v$. The two remaining directions both fail \eqref{eq:fixed}: for $\one$, the representation and $u_n(x+t\one)=u_n(x)+t$ would force $u_n(x)=x_1$ on the sector, and for $0$ the same follows from the subsolution property, since by Proposition~\ref{prop:hessian} below $u_n$ is touched from above by a quadratic at every point; both contradict $u_n>\phi_n$, which holds by \eqref{eq:pairbounds}.

We now turn to the regularity. For an open set $\Omega$, we write $\mathcal D'(\Omega)$ for the space of distributions on $\Omega$, that is, continuous linear functionals on $C_c^\infty(\Omega)$, and for distributions $f$ and $g$ we write $f\leq g$ in $\mathcal D'(\Omega)$ if $\langle f,\psi\rangle\leq\langle g,\psi\rangle$ for every nonnegative $\psi\in C_c^\infty(\Omega)$. For example, \eqref{eq:linearineq} below means that $\int u_n\,D_{\v}^2\psi\,dx\leq2\int(u_n-\phi_n)\psi\,dx$ for every nonnegative $\psi\in C_c^\infty(\R^n)$.

\begin{proposition}\label{prop:hessian}
For every binary direction $\v\in\{0,1\}^n$,
\begin{equation}\label{eq:linearineq}
 D_{\v}^2u_n\leq2(u_n-\phi_n)
 \qquad\text{in }\mathcal D'(\R^n).
\end{equation}
Consequently, $u_n\in C^{1,1}_{\mathrm{loc}}(\R^n)$ and
\begin{equation}\label{eq:roughhessian}
 0\preceq D^2u_n\preceq8(u_n-\phi_n)I\quad\text{almost everywhere in }\R^n.
\end{equation}
\end{proposition}

\begin{proof}
Fix $\v\in\{0,1\}^n$. For $\v=0$, \eqref{eq:linearineq} reads $u_n\geq\phi_n$, which follows from the supersolution property tested with a supporting affine function of the convex function $u_n$. For $\v\neq0$, since $u_n$ is a viscosity supersolution of $u_n-\tfrac12D_{\v}^2u_n\geq\phi_n$ on $\R^n$, Lemma~\ref{lem:lines}(iii) gives \eqref{eq:seconddiff} on every line in the direction $\v$. For nonnegative $\psi\in C_c^\infty(\R^n)$, a change of variables then gives
\[
 \int u_n(x)\bigl(\psi(x+h\v)+\psi(x-h\v)-2\psi(x)\bigr)dx
 \leq2h^2\int\max_{|s|\leq h}(u_n-\phi_n)(x+s\v)\,\psi(x)\,dx .
\]
Dividing by $h^2$ and letting $h\downarrow0$ gives \eqref{eq:linearineq}, by dominated convergence and the continuity of $u_n-\phi_n$.

Next we average \eqref{eq:linearineq} over $\v$. For a symmetric matrix $A$, a binary $\v$ and the sign vector $\sigma=2\v-\one$, the parallelogram law gives
\[
 \sigma^TA\sigma+\one^TA\one=2\v^TA\v+2(\one-\v)^TA(\one-\v),
\]
and when $\v$ is uniform on $\{0,1\}^n$, $\sigma$ is uniform on $\{-1,1\}^n$, so that $\E[\sigma^TA\sigma]=\tr A$. Applying this to the distributional Hessian of $u_n$, with \eqref{eq:linearineq} for $\v$ and for $\one-\v$, gives $\Delta u_n+D_{\one}^2u_n\leq8(u_n-\phi_n)$, and $D_{\one}^2u_n\geq0$ by convexity. Therefore
\begin{equation}\label{eq:laplacian}
 0\leq\Delta u_n\leq8(u_n-\phi_n)
 \qquad\text{in }\mathcal D'(\R^n),
\end{equation}
the lower bound being convexity. The distributional Hessian of a convex function is a nonnegative matrix valued measure, and \eqref{eq:laplacian} bounds its trace by a locally bounded density. Every entry is therefore absolutely continuous with locally bounded density, so $u_n\in C^{1,1}_{\mathrm{loc}}(\R^n)$, and \eqref{eq:roughhessian} follows because each eigenvalue of a nonnegative symmetric matrix is at most its trace.
\end{proof}

We emphasize the order of the argument: \eqref{eq:linearineq} is established before any regularity of $u_n$ beyond continuity and convexity, and it is what produces the regularity. It is not deduced from an almost everywhere Hessian, and it could not be: the payoff $\phi_n$ is convex with vanishing Hessian almost everywhere, so it satisfies \eqref{eq:linearineq} almost everywhere, yet it is not $C^{1,1}$.

Proposition~\ref{prop:hessian} contains the first two inequalities of Lemma~\ref{lem:regularity}. The third, $u_n-\phi_n\leq h_{n-1}(0)\,e^{-\sqrt2(x_{(1)}-x_{(2)})}$, requires an upper bound on $u_n$ with the correct dependence on the trailing experts, which comes from the global supersolution of Proposition~\ref{prop:barrier}, and we complete the proof of Lemma~\ref{lem:regularity} at the end of Section~\ref{sec:barrier}. Before that, the removal estimate \eqref{eq:pairbounds} below gives the same exponential decay with the constant $(n-1)c\geq h_{n-1}(0)$ in place of $h_{n-1}(0)$, which already shows that $u_n\in C^{1,1}(\R^n)$.

\subsection{Removing experts}

We now quantify how much the value changes when experts are removed. Removing an expert lowers the value, because the payoff $\max_i x_i$ can only decrease, while an expert trailing far behind contributes only an exponentially small amount.

Both bounds are proved by comparison. The comparison functions need not have bounded correction relative to $\phi_n$, so we first prove a comparison principle for functions of linear growth. For \eqref{eq:pde} it has an elementary proof, because a supersolution satisfies the linear inequality $W-\tfrac12D_{\v}^2W\geq\phi_n$ for every $\v$, and these inequalities survive mollification.

\begin{lemma}\label{lem:comparison}
Let $w,W\in C(\R^n)$ be a viscosity subsolution and a viscosity supersolution of \eqref{eq:pde}, respectively, and suppose that $w(x)\leq C(1+|x|)$ and $W(x)\geq-C(1+|x|)$ for some constant $C$. Then $w\leq W$ on $\R^n$.
\end{lemma}

\begin{proof}
We first suppose that $W\in C^2(\R^n)$ is a classical supersolution. For $\delta>0$ let $Q_\delta(x)=\delta(n+|x|^2)$, so that $D_{\v}^2Q_\delta=2\delta|\v|^2\leq2\delta n$ for every $\v\in\{0,1\}^n$. By the growth assumptions, $w-W-Q_\delta$ attains its maximum at some $x_*$, and $W+Q_\delta$ plus a constant touches $w$ from above there, so
\[
 w(x_*)-\frac12\max_{\v}D_{\v}^2W(x_*)-\delta n\leq\phi_n(x_*)\leq W(x_*)-\frac12\max_{\v}D_{\v}^2W(x_*).
\]
Hence $w-W-Q_\delta\leq(w-W-Q_\delta)(x_*)\leq\delta n-Q_\delta(x_*)\leq0$ everywhere, and letting $\delta\downarrow0$ gives $w\leq W$. The same argument applies when instead $w\in C^2(\R^n)$ is a classical subsolution. In particular, $w(x)=x_i$ satisfies $w-\tfrac12\max_{\v}D_{\v}^2w=x_i\leq\phi_n$, so $W\geq x_i$ for every $i$, that is, $W\geq\phi_n$.

In general, we approximate $W$ by classical supersolutions. As noted before Lemma~\ref{lem:lines}, $W$ is a viscosity supersolution of $W-\tfrac12D_{\v}^2W\geq\phi_n$ for every $\v\in\{0,1\}^n$. The argument for $\v\neq0$ in the proof of Proposition~\ref{prop:hessian} uses only this property and continuity, so it applies to $W$ and gives $D_{\v}^2W\leq2(W-\phi_n)$ in $\mathcal D'(\R^n)$ for every $\v\neq0$; for $\v=0$ this inequality is $W\geq\phi_n$. Let $\rho_\eps\geq0$ be a smooth mollifier supported in the ball of radius $\eps\leq1$, and put $W_\eps=W*\rho_\eps$. Testing these inequalities against $\rho_\eps(x-\cdot)$ gives $D_{\v}^2W_\eps\leq2(W_\eps-\phi_n*\rho_\eps)$ pointwise, for every $\v\in\{0,1\}^n$. Since $\phi_n$ is $1$-Lipschitz, $\phi_n*\rho_\eps\geq\phi_n-\eps$, and therefore
\[
 W_\eps+\eps-\frac12\max_{\v}D_{\v}^2W_\eps\geq\phi_n*\rho_\eps+\eps\geq\phi_n .
\]
Thus $W_\eps+\eps$ is a classical supersolution, and $W_\eps+\eps\geq-C(2+|x|)$. The first case gives $w\leq W_\eps+\eps$, and letting $\eps\downarrow0$ proves $w\leq W$, since $W$ is continuous.
\end{proof}

Recall that $\kappa=\sqrt2$ and $c=1/(2\kappa)$, and let
\[
 \mathcal B(t)=t_++c\,e^{-\kappa|t|},
\]
so that $\mathcal B\in C^2(\R)$ with $\mathcal B''(t)=c\kappa^2e^{-\kappa|t|}>0$ and $\mathcal B-\tfrac12\mathcal B''=t_+$. Indeed, both one-sided derivatives at $0$ are $\mathcal B'(0^\pm)=\tfrac12$ and $\mathcal B''(0^\pm)=c\kappa^2$, for $t>0$ we have $\mathcal B-\tfrac12\mathcal B''=t+c(1-\tfrac{\kappa^2}2)e^{-\kappa t}=t$, and for $t<0$ the two terms cancel identically. Lemma~\ref{lem:removal} proves the removal bounds by comparison with the one-dimensional barrier $\mathcal B$, and both are used again in the proof of Lemma~\ref{lem:prefix}. The resulting tail estimate~\eqref{eq:pairbounds} of Corollary~\ref{cor:pairbounds} is used in Sections~\ref{sec:inheritance}, \ref{sec:exist} and~\ref{sec:scalar_trace}.

\begin{lemma}\label{lem:removal}
Fix $j<n$ and an anchor index $i_*\leq j$. Then
\[
 0\leq u_n(x)-u_j(x_1,\ldots,x_j)
 \leq\sum_{i>j}\mathcal B(x_i-x_{i_*}).
\]
\end{lemma}
\begin{proof}
Write $x=(x',x'')$ with $x'\in\R^j$, and let $\widetilde u_j(x)=u_j(x')$ be the extension that is constant in the last $n-j$ coordinates. We first show that $\widetilde u_j$ is a subsolution. Let $\psi\in C^2$ touch $\widetilde u_j$ from above at $x$. Then $\psi'(y)=\psi(y,x'')$ touches $u_j$ from above at $x'$, and $D_{(\v',0)}^2\psi(x)=D_{\v'}^2\psi'(x')$ for $\v'\in\{0,1\}^j$, so
\[
 \widetilde u_j(x)-\frac12\max_{\v\in\{0,1\}^n}D_{\v}^2\psi(x)
  \leq u_j(x')-\frac12\max_{\v'\in\{0,1\}^j}D_{\v'}^2\psi'(x')
  \leq\phi_j(x')\leq\phi_n(x) .
\]
Since $\widetilde u_j$ and $u_n$ have at most linear growth, Lemma~\ref{lem:comparison} gives $\widetilde u_j\leq u_n$.

For the upper bound put $V=\widetilde u_j+\sum_{i>j}\mathcal B(x_i-x_{i_*})$. We will prove that $V$ is a supersolution. Let $\psi\in C^2$ touch $V$ from below at $x$, let $\v=(\v',\v'')\in\{0,1\}^n$, and put $\chi=\psi-\sum_{i>j}\mathcal B(x_i-x_{i_*})$, which is $C^2$ because $\mathcal B$ is. Then $\chi$ touches $\widetilde u_j$ from below at $x$, so $t\mapsto\chi(x+t\v)$ touches $t\mapsto u_j(x'+t\v')$ from below at $t=0$. If $\v'\neq0$, Lemma~\ref{lem:lines}(i), applied to $u_j$ in the direction $\v'$, gives $u_j(x')-\tfrac12D_{\v}^2\chi(x)\geq\phi_j(x')$. If $\v'=0$, the function $t\mapsto u_j(x'+t\v')$ is constant, so $D_{\v}^2\chi(x)\leq0$, and the same inequality follows from $u_j\geq\phi_j$, which is \eqref{eq:linearineq} with $\v=0$. Moreover $D_{\v}^2\mathcal B(x_i-x_{i_*})=(v_i-v_{i_*})^2\mathcal B''\leq\mathcal B''$, because $(v_i-v_{i_*})^2\in\{0,1\}$ and $\mathcal B''\geq0$. Hence
\[
 V(x)-\frac12D_{\v}^2\psi(x)
  \geq\phi_j(x')+\sum_{i>j}\Bigl(\mathcal B-\frac12\mathcal B''\Bigr)(x_i-x_{i_*})
  =\phi_j(x')+\sum_{i>j}(x_i-x_{i_*})_+ .
\]
If $\phi_n$ is attained at an index $\leq j$ the right side is at least $\phi_j(x')=\phi_n(x)$; if it is attained at $i_0>j$ then, using $\phi_j(x')\geq x_{i_*}$,
\[
 \phi_j(x')+\sum_{i>j}(x_i-x_{i_*})_+\geq x_{i_*}+(x_{i_0}-x_{i_*})=\phi_n(x) .
\]
Taking the maximum over $\v$ shows that $V$ is a supersolution. It has at most linear growth, although $V-\phi_n$ need not be bounded, so Lemma~\ref{lem:comparison} gives $u_n\leq V$.
\end{proof}

Lemma~\ref{lem:removal} gives two-sided bounds on the correction when one expert leads.

\begin{corollary}\label{cor:pairbounds}
In the coordinates $x=(s,z)$,
\begin{equation}\label{eq:pairbounds}
 c\,e^{-\kappa(s-M)}\leq u_n(s,z)-s
 \leq c\sum_{i=1}^{n-1} e^{-\kappa(s-z_i)},\qquad
 s\geq M:=\max_i z_i.
\end{equation}
\end{corollary}
\begin{proof}
For the upper bound, apply Lemma~\ref{lem:removal} with $j=1$ and $u_1(x_1)=x_1$, and use $\mathcal B(t)=c\,e^{\kappa t}$ for $t\leq0$. For the lower bound, relabel so that the second coordinate is the largest follower and apply Lemma~\ref{lem:removal} with $j=2$ and $u_2(x_1,x_2)=x_1+ce^{-\kappa(x_1-x_2)}$.
\end{proof}

The correction $u_n(s,z)-s$ is therefore squeezed between the contribution of the single best follower and the sum of the contributions of all followers, each decaying exponentially at rate $\kappa$ in the gap to the leader. Both bounds have the same exponential rate, differing only in the prefactor: one term versus $n-1$ terms. So \eqref{eq:pairbounds} pins down the decay rate $e^{-\kappa s}$ exactly.

\subsection{Prefix inheritance and five experts}\label{sec:inheritance}

We now introduce the concept of \emph{prefix inheritance}, which allows us to compare optimal adversary strategies for different numbers of experts. In particular, we prove that if $\v$ satisfies \eqref{eq:fixed} for $n$ experts, then its first $j$ bits satisfy \eqref{eq:fixed} for $j$ experts. This is used later to prove Theorem~\ref{thm:five}.
\begin{lemma}\label{lem:prefix}
Let $\v=(\v',\v'')\in\{0,1\}^n$ with $\v'\in\{0,1\}^j$, and suppose \eqref{eq:fixed} holds. Then $u_j-\tfrac12D_{\v'}^2u_j=x_1$ on $\{x_1>\cdots>x_j\}$ in the viscosity sense.
\end{lemma}
\begin{proof}
For $k>0$ set $u^{(k)}(x)=u_n(x_1,\ldots,x_j,x_{j+1}-k,\ldots,x_n-k)$. Since \eqref{eq:fixed} has constant coefficients, $u^{(k)}$ satisfies the same equation on the translated sector, which contains any fixed compact subset of $\{x_1>\cdots>x_j\}\times\{x_{j+1}>\cdots>x_n\}$ once $k$ is large. By Lemma~\ref{lem:removal} with anchor $i_*=1$,
\[
 0\leq u^{(k)}(x)-u_j(x_1,\ldots,x_j)
   \leq\sum_{i>j}\mathcal B(x_i-k-x_1)\longrightarrow0
\]
uniformly on that compact set, because $\mathcal B(t)=ce^{\kappa t}$ for $t\leq0$. So $u^{(k)}$ converges locally uniformly to the extension of $u_j$, which is constant in the last $n-j$ coordinates. Viscosity stability under locally uniform convergence passes the linear equation to the limit, and for a function independent of the last coordinates $D_{\v}^2=D_{\v'}^2$. No differentiated limit is required.
\end{proof}

Applied with $j=5$, Lemma~\ref{lem:prefix} reduces the possible directions for every $n\geq5$ to those whose five-expert prefix is globally optimal for five experts. We now prove Theorem~\ref{thm:five}, which pins that prefix down. The proof uses the symmetry of $u_5$ at the origin.

\begin{proof}[Proof of Theorem~\ref{thm:five}]
Let $\v'\in\{0,1\}^5$ have first bit one and satisfy \eqref{eq:fixed} for $n=5$. Since $u_5\in C^2(\R^5)$~\cite{calder2026explicit}, \eqref{eq:fixed} extends by continuity to the closed sector, where it says that $\v'$ attains the maximum in \eqref{eq:pde}.

We first count the ones of $\v'$. At the origin, where $\phi_5=0$,
\[
 D_{\v'}^2u_5(0)=2u_5(0)=\max_{\v\in\{0,1\}^5}D_{\v}^2u_5(0).
\]
Permutation invariance makes $D^2u_5(0)$ commute with every coordinate permutation, so $D^2u_5(0)=\lambda_5I+\mu_5\one\one^T$, and the constraint $D^2u_5(0)\one=0$ gives $\mu_5=-\lambda_5/5$. Hence, for $\v$ with $k$ ones,
\[
 D_{\v}^2u_5(0)=\lambda_5\Bigl(k-\frac{k^2}5\Bigr),
\]
which equals $\tfrac45\lambda_5,\tfrac65\lambda_5,\tfrac65\lambda_5,\tfrac45\lambda_5,0$ for $k=1,\ldots,5$ and $0$ for $k=0$. Convexity gives $\lambda_5\geq0$, so the maximum is $\tfrac65\lambda_5=2u_5(0)$, and since $u_5(0)>0$ by \eqref{eq:pairbounds}, $\lambda_5=\tfrac53u_5(0)>0$ and the maximum is attained exactly for $k\in\{2,3\}$. So \textit{$\v'$ has two or three ones}.

Next we determine the first two bits, using Lemma~\ref{lem:prefix} with $j=2$. The candidates with first bit one are $10$ and $11$. Since $D^2u_2\one=0$ we have $D_{11}^2u_2=0$, and \eqref{eq:fixed} would force $u_2=x_1$, contradicting $u_2=x_1+ce^{-\kappa(x_1-x_2)}$. Hence the first two bits are $10$, and the candidates are $10100$, $10010$, $10001$, $10110$, $10101$ and $10011$.

It remains to eliminate the last five. By~\cite[Theorem~2.2]{calder2026explicit}, the COMB direction $10101$ attains the Hamiltonian maximum exactly on $\{x_1=x_2\}\cap\{x_3=x_4\}$, a set with empty interior in the sector; so $10101$ fails \eqref{eq:fixed}. For the other four we use two symmetries of $D^2u_5$. First, $D_{\v}^2u_5=D_{\one-\v}^2u_5$ because $D^2u_5\one=0$. Second, at a point $x$ fixed by a transposition $\sigma$ we have $D^2u_5(x)=P_\sigma D^2u_5(x)P_\sigma^T$, from where $D_{\v}^2u_5(x)=D_{P_\sigma\v}^2u_5(x)$. Complementing $10010$ and then transposing the first two coordinates gives COMB, so on $\{x_1=x_2\}$ COMB would attain the maximum; choosing a point of that face with $x_3>x_4$ contradicts~\cite[Theorem~2.2]{calder2026explicit}. Transposing coordinates $3$ and $4$ turns $10011$ into COMB, which would attain the maximum on $\{x_3=x_4\}$, and a point of that face with $x_1>x_2$ gives a contradiction. In the same way, transposing coordinates $4$ and $5$ turns $10110$ into COMB, which is contradicted at a point of $\{x_4=x_5\}$ with $x_1>x_2$. Finally, transposing coordinates $4$ and $5$, complementing, and transposing the first two coordinates turns $10001$ into COMB, which is contradicted at a point of $\{x_1=x_2\}\cap\{x_4=x_5\}$ with $x_3>x_4$. Only $10100$ remains, which completes the proof. %
\end{proof}

Combining Lemma~\ref{lem:prefix} and Theorem~\ref{thm:five}, for every $n\geq5$ a direction satisfying \eqref{eq:fixed} must begin with $10100$. For $n=6$ this leaves $101000$ and $101001$; Theorem~\ref{thm:exclude}, proved in Section~\ref{sec:conclude}, removes the second, and Section~\ref{sec:last} removes the first.
\section{The one-leader profile}\label{sec:profile}

This section proves Theorem~\ref{thm:profile}, Lemma~\ref{lem:regularity}, Corollary~\ref{cor:player} and Theorem~\ref{thm:exclude}. Section~\ref{sec:algebra} collects the properties of the profile $h_m$ used in the rest of the section. Theorem~\ref{thm:profile} follows by matching an upper bound from an explicit global supersolution (Section~\ref{sec:barrier}) with a lower bound from a minimality property of the profile (Section~\ref{sec:exist}), and Section~\ref{sec:conclude} uses the resulting profile to prove Theorem~\ref{thm:exclude}.

\subsection{Properties of the profile}\label{sec:algebra}

All properties of $h_m$ below, for any $m\geq1$, are derived from its representation~\eqref{eq:profile} as an exponential moment of $\max_i(z_i+E_i)$. %
Write $M=\max_iz_i$, and set
\[
 \rho_i(t)=e^{-2\kappa(t-z_i)}\qquad(t\geq M).
\]

\begin{lemma}\label{lem:h_m}
For every $z\in\R^m$,
\begin{equation}\label{eq:integral}
 h_m(z)=\frac{e^{\kappa M}}{4\kappa}
 +\frac14\int_M^\infty e^{\kappa t}
       \left[1-\prod_{i=1}^m(1-\rho_i(t))\right]\dd.
\end{equation}
If $z_1\geq\cdots\geq z_m$, the profile is the finite sum
\begin{equation}\label{eq:sum}
 h_m(z)=\frac{e^{\kappa z_1}}{2\kappa}
 +\frac1{2\kappa}\sum_{\varnothing\ne J\subset\{2,\ldots,m\}}
 \frac{(-1)^{|J|+1}}{4|J|^2-1}
 \exp\!\left(2\kappa\sum_{i\in J}z_i-(2|J|-1)\kappa z_1\right).
\end{equation}
In particular,
\begin{equation}\label{eq:originprofile}
 h_m(0)=\frac{(2m)!!}{4\kappa(2m-1)!!}.
\end{equation}
\end{lemma}

\begin{proof}
Write $W=\max_i(z_i+E_i)$. For $t\geq M$ we have $\P(E_i\leq t-z_i)=1-\rho_i(t)$, so $\P(W\leq t)=\prod_{i=1}^m(1-\rho_i(t))$. Since $W\geq M$ almost surely, $e^{\kappa W}=e^{\kappa M}+\int_M^\infty\kappa e^{\kappa t}\one_{\{W>t\}}\,dt$, and taking expectations gives the tail-integral formula $\E e^{\kappa W}=e^{\kappa M}+\kappa\int_M^\infty e^{\kappa t}\P(W>t)\dd$. Dividing by $4\kappa$ gives \eqref{eq:integral}.

Next we expand the product in \eqref{eq:integral}. By definition,
\[
 1-\prod_{i=1}^m(1-\rho_i(t))
  =\sum_{\varnothing\ne S\subseteq\{1,\ldots,m\}}(-1)^{|S|+1}
      \prod_{i\in S}\rho_i(t),\qquad
 \prod_{i\in S}\rho_i(t)
  =\exp\Bigl(2\kappa\sum_{i\in S}z_i-2\kappa|S|t\Bigr),
\]
and each term integrates to
\[
 \frac14\int_M^\infty e^{\kappa t}\prod_{i\in S}\rho_i(t)\dd
 =\frac1{4\kappa(2|S|-1)}
    \exp\Bigl(2\kappa\sum_{i\in S}z_i-(2|S|-1)\kappa M\Bigr).
\]
Assume $z_1=M$ and split the subsets according to whether they contain the index $1$: write $S=J$ or $S=\{1\}\cup J$ with $J\subseteq\{2,\ldots,m\}$. Using $z_1=M$, both choices produce the same exponential $\exp(2\kappa\sum_{i\in J}z_i-(2|J|-1)\kappa z_1)$, with signs $(-1)^{|J|+1}$ and $(-1)^{|J|}$ and denominators $2|J|-1$ and $2|J|+1$. For $|J|=p\geq1$ they combine into
\[
 \frac{(-1)^{p+1}}{4\kappa}\left(\frac1{2p-1}-\frac1{2p+1}\right)
  =\frac{(-1)^{p+1}}{2\kappa(4p^2-1)},
\]
while $J=\varnothing$ contributes $e^{\kappa z_1}/(4\kappa)$, which together with the first term of \eqref{eq:integral} gives $e^{\kappa z_1}/(2\kappa)$. This is \eqref{eq:sum}. At $z=0$, $W$ is the maximum of $m$ i.i.d. exponential random variables with rate $2\kappa$, so $\P(W\leq t)=\left(1-e^{-2\kappa t}\right)^m$ and the standard order-statistic representation gives $W\overset{d}{=}\sum_{j=1}^m X_j$, where the $X_j$ are independent exponential random variables with rates $2\kappa,\;4\kappa,\;\ldots,\;2m\kappa$. Consequently,
\[
\E\!\left[e^{\kappa W}\right]
=\prod_{j=1}^m\E\!\left[e^{\kappa X_j}\right]
=\prod_{j=1}^m\frac{2j\kappa}{2j\kappa-\kappa}
=\prod_{j=1}^m\frac{2j}{2j-1}
=\frac{(2m)!!}{(2m-1)!!}.
\]
Therefore, we obtain
\[
h_m(0)=\frac{(2m)!!}{4\kappa(2m-1)!!}.\qedhere
\]
\end{proof}

We next collect the properties of $h_m$ used in the rest of the paper.

\begin{lemma}\label{lem:profile}
Let $m\geq1$ and $M=\max_iz_i$. The profile $h_m$ is symmetric, nondecreasing in each coordinate, and belongs to $C^2(\R^m)$. Moreover, the following hold.
\begin{enumerate}
\item[(i)] For all $z\in\R^m$ and $a\in\R$,
\begin{equation}\label{eq:boundsprofile}
 h_m(z+a\one)=e^{\kappa a}h_m(z),\qquad
 c\,e^{\kappa M}\leq h_m(z)\leq h_m(0)e^{\kappa M},
\end{equation}
and sending one coordinate to $-\infty$ removes it,
\begin{equation}\label{eq:onecoord}
 \lim_{z_m\to-\infty}h_m(z)=h_{m-1}(z_1,\ldots,z_{m-1})\quad(m\geq2),\qquad h_1(z_1)=c\,e^{\kappa z_1}.
\end{equation}
\item[(ii)] The mixed second derivatives are negative, $(h_m)_{ij}<0$ for $i\neq j$, and, with $\mathcal L_i=\tfrac12\partial_{ii}-\kappa\partial_i$,
\begin{equation}\label{eq:diag}
 2\mathcal L_ih_m=
 \begin{cases}
 -\dfrac\kappa2 e^{\kappa M}
       \displaystyle\prod_{\ell\ne i}\bigl(1-\rho_\ell(M)\bigr),
       &\text{if $z_i=M$ is the unique maximum},\\[4pt]
 0,&\text{otherwise}.
 \end{cases}
\end{equation}
\item[(iii)] For a set $J$ of follower indices, let $D_J=\sum_{i\in J}\partial_i$. Then
\begin{equation}\label{eq:deficit}
 D_J^2h_m-2\kappa D_Jh_m\leq0,
\end{equation}
with strict inequality if $|J|\geq2$, or if $J=\{i\}$ and $z_i$ is the unique maximum, and with equality otherwise.
\end{enumerate}
\end{lemma}

\begin{proof}
Symmetry, monotonicity and the homogeneity in \eqref{eq:boundsprofile} are immediate from \eqref{eq:profile}, since adding $a$ to every $z_i$ adds $a$ to $\max_i(z_i+E_i)$. If $z_{i^*}=M$, then $\max_i(z_i+E_i)\geq M+E_{i^*}$ and $\E e^{\kappa E_{i^*}}=2\kappa/(2\kappa-\kappa)=2$, so $h_m(z)\geq e^{\kappa M}/(2\kappa)=c\,e^{\kappa M}$. Monotonicity and homogeneity give $h_m(z)\leq h_m(M\one)=h_m(0)e^{\kappa M}$. For \eqref{eq:onecoord}, $\max_i(z_i+E_i)$ decreases to $\max_{i<m}(z_i+E_i)$ as $z_m\to-\infty$, so monotone convergence gives the first identity, and $h_1(z_1)=\tfrac1{4\kappa}e^{\kappa z_1}\E e^{\kappa E_1}=c\,e^{\kappa z_1}$.

We next compute the derivatives of $h_m$. Differentiate \eqref{eq:integral} in $z_i$. The endpoint $M$ depends on $z_i$ only when $z_i=M$, and then $\partial_{z_i}M=1$; since $\rho_i(M)=1$, the bracket in \eqref{eq:integral} equals $1$ at $t=M$, and the endpoint contribution $-e^{\kappa M}/4$ cancels the derivative $e^{\kappa M}/4$ of the first term. Under the integral, $\partial_{z_i}\rho_i=2\kappa\rho_i$, so
\begin{equation}\label{eq:first}
 (h_m)_i=\frac\kappa2\int_M^\infty e^{\kappa t}\rho_i(t)
                   \prod_{\ell\ne i}(1-\rho_\ell(t))\dd.
\end{equation}
Differentiating \eqref{eq:first} in $z_j$ with $j\ne i$ produces no endpoint term: at $t=M$ the surviving product contains the factor $1-\rho_j(M)=0$ whenever $j$ attains the maximum. Hence
\begin{equation}\label{eq:cross}
 (h_m)_{ij}=-\kappa^2\int_M^\infty e^{\kappa t}\rho_i(t)\rho_j(t)
                   \prod_{\ell\ne i,j}(1-\rho_\ell(t))\dd<0
                   \qquad(i\ne j).
\end{equation}
Differentiating \eqref{eq:first} in $z_i$ itself, the integral term reproduces $2\kappa(h_m)_i$, and the endpoint term survives exactly when $z_i=M$ is the unique maximum. Since $2\mathcal L_ih_m=(h_m)_{ii}-2\kappa(h_m)_i$, this gives \eqref{eq:diag}. The two branches agree at a tie, because the displayed product then contains a vanishing factor. The right sides of \eqref{eq:first}, \eqref{eq:cross} and \eqref{eq:diag} are continuous in $z$ on all of $\R^m$, the integrands being continuous and dominated by $e^{\kappa t}\rho_i(t)$. Consequently $h_m\in C^2(\R^m)$, with no appeal to a differentiated ordering map: the function \eqref{eq:profile} is intrinsically symmetric, and \eqref{eq:sum} is merely its expression on one ordered chamber.

Finally, \eqref{eq:cross} and \eqref{eq:diag} give
\[
 D_J^2h_m-2\kappa D_Jh_m
 =\sum_{i\in J}\bigl((h_m)_{ii}-2\kappa(h_m)_i\bigr)
       +2\sum_{i<j,\ i,j\in J}(h_m)_{ij}.
\]
Each term of the first sum is nonpositive by \eqref{eq:diag}, and negative only when $z_i$ is the unique maximum, while each term of the second sum is negative by \eqref{eq:cross}. This proves \eqref{eq:deficit} together with the cases of strict inequality and equality.
\end{proof}

\begin{remark}\label{rem:affine}
The second branch of \eqref{eq:diag} can also be seen directly. For $i$ not the unique maximizer, the integral representation \eqref{eq:integral} is an affine function of $e^{2\kappa z_i}$ with coefficients independent of $z_i$, and $\mathcal L_i$ annihilates both $1$ and $e^{2\kappa z_i}$ because $\tfrac12(2\kappa)^2-\kappa(2\kappa)=0$. We use this form again in the proof of Lemma~\ref{lem:minimal}.
\end{remark}

\begin{remark}[The full follower collision]\label{rem:collision}
Let $m\geq2$. At the full follower collision $z=0$ all coordinates are tied. Differentiating $h_m(z+a\one)=e^{\kappa a}h_m(z)$ in $a$ at $a=0$ gives $\sum_i(h_m)_i=\kappa h_m$, and differentiating once more in $z_i$ gives $\sum_j(h_m)_{ij}=\kappa(h_m)_i$. Symmetry makes all $(h_m)_i(0)$ equal and all off-diagonal $(h_m)_{ij}(0)$ equal, and the second branch of \eqref{eq:diag} gives $(h_m)_{ii}(0)=2\kappa(h_m)_i(0)$. Solving the three relations in turn,
\[
 (h_m)_i(0)=\frac{\kappa h_m(0)}m,\qquad
 (h_m)_{ii}(0)=\frac{2\kappa^2h_m(0)}m,\qquad
 (h_m)_{ij}(0)=-\frac{\kappa^2h_m(0)}{m(m-1)}\quad(i\ne j).
\]
Thus, for $k=|J|$,
\[
 2\kappa D_Jh_m(0)-D_J^2h_m(0)
 =\frac{2\kappa^2h_m(0)k}{m}-\frac{2\kappa^2h_m(0)k}{m}
   +\frac{\kappa^2h_m(0)k(k-1)}{m(m-1)}
 =\frac{\kappa^2h_m(0)k(k-1)}{m(m-1)}.
\]
For six experts and two selected followers, $m=5$ and $k=2$, this is $2\cdot\tfrac{64}{63\sqrt2}\cdot\tfrac2{20}=64/(315\sqrt2)$.
\end{remark}

\subsection{Upper bound: a global supersolution}\label{sec:barrier}

On the set with a unique leading coordinate, define
\begin{equation}\label{eq:barrier}
 \bar u_n(x)=x_{(1)}+e^{-\kappa x_{(1)}}
             h_{n-1}(x_{(2)},\ldots,x_{(n)}),
\end{equation}
and use the same formula at ties. Here, as in Lemma~\ref{lem:regularity}, $x_{(1)}\geq\cdots\geq x_{(n)}$ is the decreasing rearrangement of $x$. Since $h_{n-1}$ is symmetric, this is a well defined continuous, permutation invariant, translation equivariant function, and \eqref{eq:boundsprofile} bounds its correction: $0<\bar u_n-\phi_n\leq h_{n-1}(0)e^{\kappa(x_{(2)}-x_{(1)})}\leq h_{n-1}(0)$. As we establish below, $\bar u_n$ is a viscosity supersolution of \eqref{eq:pde}, and can thus be used to bound the value function.

\begin{proposition}\label{prop:barrier}
For every $n\geq2$, the function $\bar u_n$ is a viscosity supersolution of \eqref{eq:pde} on $\R^n$, and it solves \eqref{eq:pde} classically on the open set where the largest coordinate is unique. For $n=2,3$, $\bar u_n=u_n$. Consequently, $u_n\leq\bar u_n$ on $\R^n$, that is,
\begin{equation}\label{eq:upper}
 u_n(s,z)-s\leq e^{-\kappa s}h_{n-1}(z)\qquad(s\geq\max z_i).
\end{equation}
\end{proposition}

\begin{proof}
We first compute the residual where the leader is unique. Suppose $x_1$ is the unique largest coordinate and write $z=(x_2,\ldots,x_n)$, so that $\bar u_n=x_1+e^{-\kappa x_1}h_{n-1}(z)$ near $x$. Complement a binary direction if necessary so that it selects the leader, and let $J$ index the selected followers. From
\[
 \partial_{x_1}\bar u_n=1-\kappa e^{-\kappa x_1}h_{n-1},\qquad
 \partial_{x_1}^2\bar u_n=\kappa^2e^{-\kappa x_1}h_{n-1},\qquad
 \partial_{x_1}\partial_{z_i}\bar u_n=-\kappa e^{-\kappa x_1}(h_{n-1})_i,
\]
and $\partial_{z_i}\partial_{z_j}\bar u_n=e^{-\kappa x_1}(h_{n-1})_{ij}$, we get
\[
 D_{\v}^2\bar u_n=e^{-\kappa x_1}
   \bigl[\kappa^2h_{n-1}-2\kappa D_Jh_{n-1}+D_J^2h_{n-1}\bigr],
\]
so that, using $\phi_n=x_1$ and $\kappa^2=2$,
\begin{equation}\label{eq:barrierresidual}
\begin{aligned}
 \bar u_n-\frac12D_{\v}^2\bar u_n-\phi_n
 &=e^{-\kappa x_1}\Bigl[\Bigl(1-\frac{\kappa^2}2\Bigr)h_{n-1}
     +\kappa D_Jh_{n-1}-\frac12D_J^2h_{n-1}\Bigr]\\
 &=\frac12e^{-\kappa x_1}
       (2\kappa D_Jh_{n-1}-D_J^2h_{n-1})\geq0
\end{aligned}
\end{equation}
by \eqref{eq:deficit}. The zeroth order term drops out exactly because $\kappa^2=2$. There is equality for the leader alone, $J=\varnothing$, so
\[
 \bar u_n-\frac12\max_{\v}D_{\v}^2\bar u_n-\phi_n
  =\min_{\v}\bigl(\bar u_n-\frac12D_{\v}^2\bar u_n-\phi_n\bigr)=0 .
\]
Hence $\bar u_n$ solves the Bellman PDE classically on the open set with a unique leader, including follower collisions, where $h_{n-1}\in C^2$.

It remains to check the top collisions. Suppose the top two coordinates are equal, so that the leading follower satisfies $z_1=M=x_{(1)}$. Set
\[
 \theta_j=e^{2\kappa(z_j-M)}\in(0,1]\ (j=1,\ldots,n-1),\qquad
 \theta_1=1,\qquad
 P(\rho)=\prod_{j=2}^{n-1}(1-\theta_j\rho),
\]
and substitute $\rho=e^{-2\kappa(t-M)}$ in \eqref{eq:integral} and \eqref{eq:first}, so that $\rho_j(t)=\theta_j\rho$, $e^{\kappa t}=e^{\kappa M}\rho^{-1/2}$ and $\dd=-d\rho/(2\kappa\rho)$. Since $\prod_{j=1}^{n-1}(1-\rho_j)=(1-\rho)P(\rho)$,
\[
 \kappa h_{n-1}=e^{\kappa M}\left(\frac14
   +\frac18\int_0^1\rho^{-3/2}\bigl[1-(1-\rho)P(\rho)\bigr]d\rho\right),
 \quad
 (h_{n-1})_1=\frac{e^{\kappa M}}4\int_0^1\rho^{-1/2}P(\rho)\,d\rho.
\]
Both integrals converge, since $1-(1-\rho)P(\rho)=O(\rho)$ at the origin. Writing $\tfrac14\rho^{-1/2}P=\tfrac18\rho^{-3/2}\cdot2\rho P$ and using
\[
 1-(1-\rho)P+2\rho P=1+(3\rho-1)P=(3\rho-1)(P-1)+3\rho,
 \qquad \frac18\int_0^13\rho^{-1/2}d\rho=\frac34,
\]
the constants $\tfrac14+\tfrac34-1$ cancel and
\begin{equation}\label{eq:cusp}
 e^{-\kappa M}(\kappa h_{n-1}+(h_{n-1})_1)-1
 =\frac18\int_0^1\rho^{-1/2}(3\rho-1)\,\frac{P(\rho)-1}{\rho}\,d\rho .
\end{equation}
The sign of the right side follows from Chebyshev's integral inequality. The kernel $k(\rho)=\rho^{-1/2}(3\rho-1)$ changes sign once, at $\rho=\tfrac13$, and $\int_0^1k\,d\rho=2-2=0$. The function $g(\rho)=(P(\rho)-1)/\rho$ is the slope of the chord of $P$ from $0$ to $\rho$, and it is nondecreasing because $P$ is convex on $[0,1]$. Indeed,
\[
 P''(\rho)=\sum_{\substack{2\leq j,k\leq n-1\\ j\ne k}}\theta_j\theta_k
    \prod_{\substack{2\leq\ell\leq n-1\\ \ell\ne j,k}}(1-\theta_\ell\rho)\geq0
    \quad\text{on }[0,1],
\]
since every $\theta_\ell\in(0,1]$ makes each factor nonnegative for $\rho\leq1$. Hence
\[
 \int_0^1k\,g\,d\rho=\int_0^1k(\rho)\bigl(g(\rho)-g(\tfrac13)\bigr)\,d\rho\geq0,
\]
because both factors have the sign of $\rho-\tfrac13$. For $n\geq4$ each product in $P''$ is positive on $(0,1)$, so $g$ is strictly increasing and the inequality is strict. For $n\leq3$ the polynomial $P$ is affine, so $g$ is constant and the left side of \eqref{eq:cusp} vanishes.

The derivative of $\bar u_n$ into the sector along the difference of the two top coordinates is
\[
 (\bar u_n)_1-(\bar u_n)_2
 =1-e^{-\kappa M}(\kappa h_{n-1}+(h_{n-1})_1)<0
 \qquad(n\geq4).
\]

Along the line $\zeta\mapsto x+\zeta(e_1-e_2)$ the symmetric extension is even in $\zeta$, so it has the form $\bar u_n(x)-\gamma_*|\zeta|+O(\zeta^2)$ with $\gamma_*>0$: a strict downward cusp. This remains true if more than two top coordinates are tied, since the tied slot still has $\theta_2=1$ and $P''>0$ for $n\geq4$. A differentiable $\psi$ with $\bar u_n-\psi$ minimal at $x$ would satisfy $\psi'(0^+)\leq-\gamma_*$ and $\psi'(0^-)\geq\gamma_*$ along that line, which is impossible. Hence no test function touches $\bar u_n$ from below at a top collision, the viscosity supersolution condition there is vacuous, and together with \eqref{eq:barrierresidual} this proves that $\bar u_n$ is a global viscosity supersolution for $n\geq4$.

For $n=2,3$ the left side of \eqref{eq:cusp} vanishes and the barrier is the solution itself. Indeed \eqref{eq:sum} gives $h_1(z_1)=e^{\kappa z_1}/(2\kappa)$ and $h_2(z_1,z_2)=e^{\kappa z_1}/(2\kappa)+e^{2\kappa z_2-\kappa z_1}/(6\kappa)$, so that
\[
 \bar u_2=x_1+\frac{e^{\kappa(x_2-x_1)}}{2\kappa},\qquad
 \bar u_3=x_1+\frac{e^{\kappa(x_2-x_1)}}{2\kappa}
        +\frac{e^{\kappa(2x_3-x_2-x_1)}}{6\kappa},
\]
which are exactly the two and three expert solutions~\cite[(3.2),(3.3)]{calder2026explicit}; direct differentiation verifies the supersolution assertion, with $C^2$ matching across the top faces. In particular, the normalization in \eqref{eq:profile}, for which $h_1(z_1)=c\,e^{\kappa z_1}$ with $c=1/(2\kappa)$, is the one forced by the two-expert solution.

Finally, $u_n$ and $\bar u_n$ have at most linear growth, so Lemma~\ref{lem:comparison} gives $u_n\leq\bar u_n$, which is \eqref{eq:upper}.
\end{proof}

For $n\geq4$, $\bar u_n$ is only an upper barrier globally: its downward cusps are incompatible with convexity of the true value. Its usefulness is that it has the correct leading tail.

We can now complete our regularity proof. 
\begin{proof}[Proof of Lemma~\ref{lem:regularity}]
By Proposition~\ref{prop:hessian}, $u_n\in C^{1,1}_{\mathrm{loc}}(\R^n)$ and $0\preceq D^2u_n\preceq8(u_n-\phi_n)I$ almost everywhere. Given $x$, we apply \eqref{eq:upper} with $s=x_{(1)}$ and $z=(x_{(2)},\dots,x_{(n)})$, which is legitimate by permutation invariance, and then \eqref{eq:boundsprofile} with $\max_iz_i=x_{(2)}$, to obtain
\[
 u_n(x)-\phi_n(x)\leq e^{-\kappa x_{(1)}}h_{n-1}(x_{(2)},\dots,x_{(n)})\leq h_{n-1}(0)\,e^{-\kappa(x_{(1)}-x_{(2)})}.
\]
This is the third inequality of the lemma. Since $x_{(1)}\geq x_{(2)}$, it also bounds the Hessian of $u_n$ by $8h_{n-1}(0)$ almost everywhere, so $\nabla u_n$ is Lipschitz on $\R^n$ with constant $8h_{n-1}(0)$ and $u_n\in C^{1,1}(\R^n)$.
\end{proof}

\subsection{Lower bound: existence and minimality of the limit}\label{sec:exist}

We show that the rescaled correction $e^{\kappa s}(u_n(s,z)-s)$ converges to a limit $h_\star\le h_{n-1}$ that is concave in each scale variable $e^{2\kappa z_i}$, which is the form the coordinatewise inequalities $\mathcal L_ih_\star\le0$ take here (Lemmas~\ref{lem:limit} and~\ref{lem:limitineq}). We then show that every function with these properties of $h_\star$ dominates $h_{n-1}$ (Lemma~\ref{lem:minimal}). Together these prove Theorem~\ref{thm:profile}. Apart from the bounds \eqref{eq:upper} and \eqref{eq:pairbounds}, Lemmas~\ref{lem:limit} and~\ref{lem:limitineq} use \eqref{eq:pde} only through the one-dimensional comparison of Lemma~\ref{lem:lines}(ii) along the binary directions $e_1$ and $e_1+e_{i+1}$, and Lemma~\ref{lem:minimal} is a statement about functions.

\begin{lemma}\label{lem:limit}
For each $z\in\R^{n-1}$, the function $s\mapsto e^{\kappa s}(u_n(s,z)-s)$ is nondecreasing on $(\max_iz_i,\infty)$. As $s\to\infty$ it converges, locally uniformly in $z$, to a continuous, convex and permutation invariant function $h_\star$ with $h_\star(z+a\one)=e^{\kappa a}h_\star(z)$ and
\begin{equation}\label{eq:hbounds}
 c e^{\kappa\max z_i}\leq h_\star(z)\leq h_{n-1}(z).
\end{equation}
\end{lemma}

\begin{proof}
Fix $z$, let $M=\max_iz_i$, and write $g(s)=u_n(s,z)$ and $\chi(s)=g(s)-s$ for $s>M$. By \eqref{eq:pairbounds}, $\chi\geq0$. Along the line $s\mapsto(s,z)$, which has the binary direction $e_1$, the payoff is $\phi_n(s,z)=s$ for $s\geq M$, so Lemma~\ref{lem:lines}(ii) applies to $g$ with $\tilde f(s)=s$ on every interval $[s_1,L]\subset(M,\infty)$. Let $M<s_1<s$ and $\eps>0$, and put
\[
 G_\eps(t)=\chi(s_1)e^{-\kappa(t-s_1)}-\eps\bigl(e^{\kappa(t-s_1)}-e^{-\kappa(t-s_1)}\bigr).
\]
Since $\kappa^2=2$, the function $H(t)=t+G_\eps(t)$ satisfies $H-\tfrac12H''=t$, and $H(s_1)=g(s_1)$, while $H(L)<L\leq g(L)$ for $L$ large because $\chi\geq0$. Hence $g\geq H$ on $[s_1,L]$, that is, $\chi\geq G_\eps$ there, and letting $\eps\downarrow0$ gives $\chi(s)\geq\chi(s_1)e^{-\kappa(s-s_1)}$. Equivalently, $e^{\kappa s}\chi(s)\geq e^{\kappa s_1}\chi(s_1)$, which is the claimed monotonicity.

The upper bound \eqref{eq:upper} then proves existence of the limit $h_\star(z)$, and together with the lower bound in \eqref{eq:pairbounds} it gives \eqref{eq:hbounds}. For fixed $s$, $e^{\kappa s}(u_n(s,z)-s)$ is convex in $z$, being the restriction of a convex function to an affine slice. The local bounds imply local equicontinuity of these convex functions, so the convergence is locally uniform and $h_\star$ is continuous, convex, and permutation invariant. Translation equivariance gives $h_\star(z+a\one)=e^{\kappa a}h_\star(z)$.
\end{proof}

For $z\in\R^{n-1}$, an index $i$ and $t\in\R$, write $z^{(t)}$ for $z$ with $z_i$ replaced by $t$, and $V_s(z)=e^{\kappa s}(u_n(s,z)-s)$.

\begin{lemma}\label{lem:limitineq}
The limit $h_\star$ of Lemma~\ref{lem:limit} is concave in each scale variable $\xi_i=e^{2\kappa z_i}$: for every $i$, $z$ and $a<t<b$,
\begin{equation}\label{eq:limitineq}
 h_\star(z^{(t)})\geq\frac{(e^{2\kappa b}-e^{2\kappa t})\,h_\star(z^{(a)})+(e^{2\kappa t}-e^{2\kappa a})\,h_\star(z^{(b)})}{e^{2\kappa b}-e^{2\kappa a}} .
\end{equation}
Since $\mathcal L_i=\tfrac12\partial_{ii}-\kappa\partial_i=2\kappa^2\xi_i^2\partial_{\xi_i}^2$, this is the inequality $\mathcal L_ih_\star\leq0$.
\end{lemma}

\begin{proof}
Fix $i$, $z$ and $a<t<b$, and let $S$ be large. Along the segment $r\mapsto(S+r,z^{(r)})$, $a\leq r\leq b$, which moves the leader and follower $i$ together in the binary direction $e_1+e_{i+1}$, the leader stays ahead, so the payoff is $S+r$. Lemma~\ref{lem:lines}(ii) therefore applies to $g(r)=u_n(S+r,z^{(r)})$ with $\tilde f(r)=S+r$. Let $H(r)=S+r+Ae^{\kappa r}+Be^{-\kappa r}$ with $H(a)=g(a)$ and $H(b)=g(b)$. Since $H-\tfrac12H''=S+r$, we get $g\geq H$ on $[a,b]$, and multiplying by $e^{\kappa(S+r)}$,
\[
 V_{S+r}(z^{(r)})\geq e^{\kappa S}\bigl(Ae^{2\kappa r}+B\bigr)\qquad(a\leq r\leq b).
\]
The right side is affine in $e^{2\kappa r}$ and equals $V_{S+a}(z^{(a)})$ and $V_{S+b}(z^{(b)})$ at $r=a$ and $r=b$, so at $r=t$ this is \eqref{eq:limitineq} with $h_\star$ replaced by $V_{S+t}$, $V_{S+a}$ and $V_{S+b}$. Letting $S\to\infty$ and using Lemma~\ref{lem:limit} at the three points proves \eqref{eq:limitineq}.
\end{proof}

It remains to show that these properties and the lower bound in \eqref{eq:hbounds} force $h_\star\geq h_{n-1},$ which would prove Theorem~\ref{thm:profile}. This is the substantive identification step; neither homogeneity nor the limiting Bellman identity alone would determine the profile. The identification step is the following minimality lemma, which we apply with $m=n-1$ to $w=h_\star$. Its proof compares $w$ with $h_m$ at the vertices of the scale variables: $h_m$ is affine in each scale variable on the set where that coordinate is not the unique maximum (Remark~\ref{rem:affine}), while $w$ is concave there, and at the full ties convexity and homogeneity reproduce the recursion behind \eqref{eq:originprofile}.

\begin{lemma}\label{lem:minimal}
Let $m\geq1$, and let $w:\R^m\to\R$ be a convex, permutation invariant function that is concave in each scale variable $\xi_i=e^{2\kappa z_i}$, in the sense of \eqref{eq:limitineq}, and satisfies
\[
 w(z+a\one)=e^{\kappa a}w(z)\quad\text{and}\quad w(z)\geq c e^{\kappa\max_i z_i}\qquad\text{for }z\in\R^m,\ a\in\R.
\]
Then $w\geq h_m$.
\end{lemma}

\begin{proof}
We argue by induction on $m$. For $m=1$, the lower bound and \eqref{eq:onecoord} give $w(z)\geq c\,e^{\kappa z}=h_1(z)$. Let $m\geq2$ and assume the lemma in dimension $m-1$.

We first remove a coordinate. For fixed $z_1,\ldots,z_{m-1}$, the function $\xi_m\mapsto w(z)$ is concave and positive on $(0,\infty)$, hence nondecreasing, since a concave function that decreases somewhere tends to $-\infty$. So the limit
\[
 \underline w(z_1,\ldots,z_{m-1})=\lim_{z_m\to-\infty}w(z)
\]
exists. The function $\underline w$ is again convex, permutation invariant and concave in each scale variable; it satisfies $\underline w(y+a\one)=e^{\kappa a}\underline w(y)$, because $z_m+a\to-\infty$ as well, and $\underline w(y)\geq c e^{\kappa\max_iy_i}$. By the induction hypothesis, $\underline w\geq h_{m-1}$. By symmetry, sending any other coordinate to $-\infty$ gives $\underline w$ of the remaining coordinates, and by \eqref{eq:onecoord} and symmetry the corresponding limit of $h_m$ is $h_{m-1}$.

Next we treat the full tie. The $m$ points $\rho e_1,\ldots,\rho e_m$ average to $(\rho/m)\one$, so convexity, symmetry and homogeneity give $w(\rho e_m)\geq e^{\kappa\rho/m}w(0)$ for every $\rho\in\R$. In the variable $\xi=e^{2\kappa\rho}$, the function $f(\xi)=w(\rho e_m)$ is concave and nondecreasing on $(0,\infty)$, with $f(0+)=\underline w(0)$ and $f(\xi)\geq\xi^{1/(2m)}f(1)$, with equality at $\xi=1$. The one-sided derivatives at $\xi=1$ therefore satisfy $f'(1^-)\leq f(1)/(2m)\leq f'(1^+)$, while concavity gives $f'(1^+)\leq f'(1^-)$, so $f'(1^-)=f(1)/(2m)$. By concavity, $f(\xi)\leq f(1)-f'(1^-)(1-\xi)$ for $0<\xi<1$, and letting $\xi\downarrow0$ gives $\underline w(0)\leq(1-\tfrac1{2m})w(0)$. By \eqref{eq:originprofile}, $h_m(0)=\tfrac{2m}{2m-1}h_{m-1}(0)$, so
\[
 w(0)\geq\frac{2m}{2m-1}\,\underline w(0)\geq\frac{2m}{2m-1}\,h_{m-1}(0)=h_m(0).
\]

Finally we reduce to the full tie. Fix $z$, let $M=\max_iz_i$, and argue by induction on the number of coordinates below $M$. If there are none, $z=M\one$ and homogeneity gives $w(z)=e^{\kappa M}w(0)\geq e^{\kappa M}h_m(0)=h_m(z)$. Otherwise let $z_i<M$ and $\theta=e^{2\kappa(z_i-M)}\in(0,1)$, and let $\hat z\in\R^{m-1}$ list the coordinates $z_j$, $j\neq i$. Concavity of $w$ in $\xi_i$ on $(0,e^{2\kappa M}]$, with the left endpoint sent to $0$, gives
\[
 w(z)\geq(1-\theta)\,\underline w(\hat z)+\theta\,w(z^{(M)}),
\]
where $z^{(M)}$ is $z$ with $z_i$ replaced by $M$. By Remark~\ref{rem:affine}, $h_m$ is affine in $\xi_i$ on $(0,e^{2\kappa M}]$, because $M$ is attained by another coordinate, so the same relation holds for $h_m$ with equality. Since $\underline w\geq h_{m-1}$ and $z^{(M)}$ has fewer coordinates below $M$, the induction gives $w(z)\geq(1-\theta)h_{m-1}(\hat z)+\theta h_m(z^{(M)})=h_m(z)$.
\end{proof}

\begin{proof}[Proof of Theorem~\ref{thm:profile}]
By Lemmas~\ref{lem:limit} and~\ref{lem:limitineq}, the limit $h_\star$ satisfies the hypotheses of Lemma~\ref{lem:minimal} with $m=n-1$, which therefore gives $h_\star\geq h_{n-1}$, while \eqref{eq:hbounds} gives $h_\star\leq h_{n-1}$. Hence $h_\star=h_{n-1}$. Since $h_{n-1}$ is continuous and positive, the locally uniform convergence of $e^{\kappa s}(u_n(s,z)-s)$ to $h_{n-1}(z)$ in Lemma~\ref{lem:limit} is the statement $u_n(s,z)-s\sim e^{-\kappa s}h_{n-1}(z)$, locally uniformly in $z$, and the monotonicity in $s$ is also part of Lemma~\ref{lem:limit}.
\end{proof}

\begin{proof}[Proof of Corollary~\ref{cor:player}]
For fixed $s$, the function $V_s(z)=e^{\kappa s}(u_n(s,z)-s)$ is convex and differentiable in $z$, with $\partial_{z_i}V_s(z)=e^{\kappa s}\partial_{z_i}u_n(s,z)$, so for $t>0$
\[
 \frac{V_s(z)-V_s(z-te_i)}t\leq e^{\kappa s}\partial_{z_i}u_n(s,z)\leq\frac{V_s(z+te_i)-V_s(z)}t .
\]
Let $K$ be compact and $\eps>0$. Since $h_{n-1}\in C^2$, we may choose $t>0$ such that both difference quotients of $h_{n-1}$ are within $\eps$ of $\partial_{z_i}h_{n-1}(z)$ for $z\in K$, and by Theorem~\ref{thm:profile} and Lemma~\ref{lem:limit} the difference quotients of $V_s$ are within $\eps$ of those of $h_{n-1}$ on $K$ for all large $s$. Hence $e^{\kappa s}\partial_{z_i}u_n(s,z)\to\partial_{z_i}h_{n-1}(z)$ uniformly on $K$. For the leader, $\nabla u_n\cdot\one=1$ gives $e^{\kappa s}(1-\partial_su_n)=\sum_ie^{\kappa s}\partial_{z_i}u_n\to\sum_i\partial_{z_i}h_{n-1}=\kappa h_{n-1}$, locally uniformly, the last identity by differentiating $h_{n-1}(z+a\one)=e^{\kappa a}h_{n-1}(z)$ in $a$ at $a=0$. The limits are continuous and positive, by \eqref{eq:first} for $\partial_{z_i}h_{n-1}$, so this is \eqref{eq:player}. For the tied case, by symmetry all $\partial_{z_i}h_{n-1}(0)$ are equal, so each equals $\kappa h_{n-1}(0)/(n-1)$. For six experts this is $\kappa h_5(0)/5=64/315$ by \eqref{eq:originprofile}.
\end{proof}

\begin{remark}[Agreement with the four- and five-expert solutions]\label{rem:tails}
Theorem~\ref{thm:profile} can be checked against the closed-form solutions of~\cite{calder2026explicit}. Substituting $x=(s,0,0,0)$ in the four-expert formula~\cite[(2.14)]{calder2026explicit}, where $z_1=z_2=z_3=Z:=s/\kappa$ and $\kappa s=2Z$, gives
\[
 F_4(Z,Z,Z)=\arctan(e^{-Z})\cosh^3Z
   +\frac12\log\coth\frac Z2\,\sinh^3Z-\frac12\sinh2Z .
\]
Expanding in powers of $e^{-Z}$, the coefficients of $e^{2Z}$ cancel and the coefficient of $e^{-2Z}$ is $\tfrac45-\tfrac14+\tfrac14=\tfrac45$, so
\[
 u_4(s,0,0,0)-s
  =\frac{F_4(Z,Z,Z)}{\kappa}\sim\frac4{5\sqrt2}e^{-\sqrt2s}
  =h_3(0)e^{-\sqrt2s}\quad\text{as }s\to\infty,
\]
in agreement with Theorem~\ref{thm:profile} and \eqref{eq:originprofile}. The comparison for five experts is also exact. Along $x=(s,0,0,0,0)$ the Region III coordinates of~\cite[Theorem 2.1]{calder2026explicit} are $a_1=a_2=a_3=a_4=a:=\kappa s/3$. Expanding the quadrature~\cite[(2.8)]{calder2026explicit} in powers of $e^{-2X}$ gives
\[
 e(X)=\frac12+\frac12e^{-2X}-\frac25e^{-4X}+\Bigl(\frac{12}7X-\frac{671}{490}\Bigr)e^{-6X}+O(Xe^{-8X})\quad\text{as }X\to\infty,
\]
where the term linear in $X$ comes from the term $X/16$ of $I(X)$ in~\cite[(2.7)]{calder2026explicit}. This expansion is a convergent series in $e^{-2X}$ with coefficients affine in $X$, so it may be differentiated term by term. The function $\lambda(L)=\log\coth(L/2)$ in~\cite[(2.17)]{calder2026explicit} has the convergent expansion
\[
 \lambda(a)=2\sum_{j\geq0}\frac{e^{-(2j+1)a}}{2j+1}=2e^{-a}+\frac23e^{-3a}+\frac25e^{-5a}+\frac27e^{-7a}+O(e^{-9a}).
\]
In~\cite[(2.16)--(2.17)]{calder2026explicit} with $a_1=a_2=a_3=a_4=a$, the trace and its derivatives, evaluated at $a$, are multiplied by functions of order $e^{3a}$, while $\lambda(a)$ is multiplied by $\tfrac{12}7\sinh^3a\cosh a$, which is of order $e^{4a}$. The displayed expansions therefore determine $F$ up to $O(ae^{-5a})$, and all four displayed terms of $\lambda$ are needed. Substituting them, the coefficients of $e^{3a}$, $e^{a}$ and $e^{-a}$ cancel, including those multiplying $a$, and the coefficient of $e^{-3a}$ in $F$ is $32/35$. The expansion of $e$ and these cancellations are verified in exact rational arithmetic in \texttt{five\_expert\_tail.py}. Since $3a=\sqrt2s$,
\[
 u_5(s,0,0,0,0)-s=\frac F\kappa\sim\frac{32}{35\sqrt2}e^{-\sqrt2s}=h_4(0)e^{-\sqrt2s}\quad\text{as }s\to\infty,
\]
again in agreement with Theorem~\ref{thm:profile} and \eqref{eq:originprofile}.
\end{remark}

\subsection{At most two ones: proof of Theorem~\ref{thm:exclude}}\label{sec:conclude}

With the profile identified, Theorem~\ref{thm:exclude} follows by passing the linear equation of a hypothetical optimal direction to the one-leader limit, where the sign inequality~\eqref{eq:deficit} allows it to select at most one follower.

\begin{proof}[Proof of Theorem~\ref{thm:exclude}]
Suppose $\v=(1,\mathbf w)$ satisfies \eqref{eq:fixed}, let $J$ be the set of selected followers, so that $\mathbf w=\one_J$, and suppose that $|J|\geq2$. Fix $z$ with $z_1>\cdots>z_{n-1}$, and choose $\eps>0$ such that $z+r\mathbf w$ is strictly decreasing for $|r|\leq\eps$. For $S>z_1+2\eps$, the segment $r\mapsto(S+r,z+r\mathbf w)$, $|r|\leq\eps$, lies in the ordered sector and has the direction $\v$, and the payoff along it is $S+r$. By Lemma~\ref{lem:lines}(iv),
\[
 u_n(S+r,z+r\mathbf w)=S+r+Ae^{\kappa r}+Be^{-\kappa r}\qquad(|r|\leq\eps),
\]
with $A$ and $B$ depending on $S$. Multiplying by $e^{\kappa(S+r)}$,
\[
 e^{\kappa(S+r)}\bigl(u_n(S+r,z+r\mathbf w)-(S+r)\bigr)=e^{\kappa S}\bigl(Ae^{2\kappa r}+B\bigr),
\]
which is affine in $e^{2\kappa r}$. As $S\to\infty$, the left side converges to $h_{n-1}(z+r\mathbf w)$ for each $r$, by Lemma~\ref{lem:limit} and Theorem~\ref{thm:profile}, and pointwise limits of functions affine in $e^{2\kappa r}$ are affine in $e^{2\kappa r}$, since such a function is determined by its values at $r=\pm\eps$. Hence $h_{n-1}(z+r\mathbf w)=\alpha+\beta e^{2\kappa r}$ for $|r|\leq\eps$, and since $h_{n-1}\in C^2$, differentiating at $r=0$ gives
\[
 D_J^2h_{n-1}(z)-2\kappa D_Jh_{n-1}(z)=4\kappa^2\beta-2\kappa\cdot2\kappa\beta=0
 \qquad\text{on }\{z_1>\cdots>z_{n-1}\}.
\]
By Lemma~\ref{lem:profile}(iii), the left side is strictly negative when $|J|\geq2$, a contradiction. Thus $\v$ has at most two ones. By Lemma~\ref{lem:prefix} and Theorem~\ref{thm:five} the five-expert prefix must be $10100$, which already uses both ones; every later bit is therefore zero.
\end{proof}

\begin{remark}\label{rem:sharper}
The proof gives more than the statement. By the strictness clause in Lemma~\ref{lem:profile}(iii), the singleton $J=\{1\}$, the leading follower, is also excluded. So the profile alone leaves only $10\cdots0$ and the directions with a single one in a position $\geq3$. The first step of the proof of Theorem~\ref{thm:five}, applied to the five-bit prefix through Lemma~\ref{lem:prefix}, shows that this prefix has two or three ones, which excludes $10\cdots0$ and every direction whose second one lies in position six or later. Only the choice among positions three, four and five requires the transposition argument and~\cite[Theorem~2.2]{calder2026explicit}.
\end{remark}

\section{Nonexistence for six or more experts}
\label{sec:last}

Assuming the last candidate $1010\cdots0$ is optimal, we restrict its linear equation to the sector where the bottom $q=n-3$ experts are tied (Section~\ref{sec:slice}) and show, through trace equations on the faces (Section~\ref{sec:scalar_trace}) and an explicit transfer kernel (Section~\ref{sec:transfer}), that the quadratic coefficient of $u_n$ along the three-leader ray $(s,s,s,0,\dots,0)$ must equal $\tfrac{3q}{2(q+1)}u_n(0)$. The Bellman inequality for the direction $1110\cdots0$ bounds the same coefficient by $u_n(0)$, which is a contradiction for $q\ge3$ and proves Theorem~\ref{thm:nonexistence} (Section~\ref{sec:contradiction}), while at $q=2$ the two bounds coincide, consistent with the optimality of $10100$ for five experts.
\subsection{The sector where the bottom experts are tied}\label{sec:slice}

We first reduce the equation for $1010\cdots0$ to a boundary-value problem for a function of three scaled gaps. Throughout this section $n\geq5$, $q=n-3\geq2$, and $A=(1,0,1,0,\dots,0)$. In Sections~\ref{sec:slice}--\ref{sec:transfer} we assume that $A$ satisfies \eqref{eq:fixed}; for $n\geq6$ this is the hypothesis that the proof of Theorem~\ref{thm:nonexistence} will contradict, while for $n=5$ it holds by~\cite{calder2026explicit}, which makes the case $q=2$ a consistency check. On the part of the sector where the bottom $q=n-3$ coordinates $x_4,\dots,x_n$ coincide, introduce the scaled gaps and the correction $F$,
\[
y_1=\kappa(x_1-x_2),\quad y_2=\kappa(x_2-x_3),\quad y_3=\kappa(x_3-x_4),
\]
\[
u_n(x_1,x_2,x_3,x_4,\dots,x_4)=x_1+\frac1\kappa\,F(y_1,y_2,y_3).
\]
By Section~\ref{sec:basic}, $F$ is bounded and $C^{1,1}$ up to the boundary of the positive orthant, with bounded first derivatives.

\begin{lemma}\label{lem:slice}
Suppose that $A$ satisfies \eqref{eq:fixed}. Then $F$ satisfies
\begin{equation}\label{eq:fchar}
(\partial_1-\partial_2+\partial_3)^2F=F\qquad\text{along characteristics,}
\end{equation}
that is, $F(y+\sigma(1,-1,1))=A_0e^{\sigma}+B_0e^{-\sigma}$ on every segment of the closed orthant in the direction $(1,-1,1)$, with constants $A_0$, $B_0$ depending on the segment, and the face conditions
\begin{equation}\label{eq:faces}
2F_1-F_2=-1\ \ (y_1=0),\ \ 
-F_1+2F_2-F_3=0\ \ (y_2=0),\ \ 
F_3=\alpha F_2\ \ (y_3=0),
\end{equation}
where $\alpha=\frac{q}{q+1}$.
\end{lemma}

\begin{proof}
Moving $x$ in the direction $A$ moves $y$ in the direction $\kappa(1,-1,1)$. On a segment of the open sector in the direction $A$, Lemma~\ref{lem:lines}(iv) gives $u_n(x+sA)=x_1+s+A_0e^{\kappa s}+B_0e^{-\kappa s}$, and with $\sigma=\kappa s$ and $u_n=x_1+F/\kappa$ this is $F=\kappa A_0e^{\sigma}+\kappa B_0e^{-\sigma}$ along the corresponding characteristic, which is \eqref{eq:fchar}. No Hessian on the bottom collision is needed: $A$ has no components in the bottom coordinates, so these representations hold on the affine sections $\{(x_1,x_2,x_3,x_4-\eps_1,\dots,x_4-\eps_q):x_1>x_2>x_3>x_4\}$, $0=\eps_1<\dots<\eps_q$, which lie in the open sector. As $\eps\downarrow0$ the restrictions converge uniformly because $u_n$ is Lipschitz, and uniform limits of functions of the form $A_0e^{\sigma}+B_0e^{-\sigma}$ on a segment have the same form, the coefficients being determined by two values. This gives \eqref{eq:fchar} on the open orthant, and on the closed orthant by continuity.

The derivatives of $u_n$ in its first three coordinates are $1+F_1$, $-F_1+F_2$ and $-F_2+F_3$, and each of the $q$ equal bottom coordinates carries $-F_3/q$. Since $u_n$ is $C^1$ and symmetric, $\partial_{x_i}u_n=\partial_{x_{i+1}}u_n$ on $\{x_i=x_{i+1}\}$, which gives \eqref{eq:faces}.
\end{proof}

Equation~\eqref{eq:fchar} is~\cite[(3.5)]{calder2026explicit} with one gap removed, and the first two face conditions are~\cite[(3.6)]{calder2026explicit}; only the third sees the cluster size.

\subsection{A scalar trace is twice differentiable by its ODE}\label{sec:scalar_trace}

This subsection extracts a single one-variable function $e(X)$ from the three-dimensional boundary-value problem and shows that its derivative has a first-order expansion at the origin, without assuming any second derivative of $u_n.$

Set $\Psi(\sigma,\omega)=F(\sigma,\omega,\sigma)$, $X=\sigma+\omega$, $H(X)=\Psi(X,0)$ and $e(X)=\Psi(0,X)$. Thus $e(X)=F(0,X,0)$ is the correction on the two-leader ray $x_1=x_2>x_3=x_4=\dots=x_n$. For $n=5$ it is the trace $e$ of~\cite{calder2026explicit}, given there by the quadrature~\cite[(2.8)]{calder2026explicit}; for $n\geq6$ the same construction applies to $u_n$, and only the constant $\beta$ in \eqref{eq:psi} below changes. The next lemma collects what the rest of the argument needs about $e$.

\begin{lemma}\label{lem:trace}
Suppose that $A$ satisfies \eqref{eq:fixed}. Then $e\in C^1([0,\infty))\cap C^2((0,\infty))$, its derivative $e'(X)=F_2(0,X,0)$ is bounded and Lipschitz, and
\begin{equation}\label{eq:e0}
e'(X)=e'(0)+e(0)\,X+O(X^2)\quad(X\downarrow0),\qquad e'(0)=-\frac{q+1}{q+3},\qquad
e(0)=\kappa u_n(0)>0.
\end{equation}
\end{lemma}

\begin{proof}
We first compute $e'(0)$ from the face conditions alone. At the corner $y=0$ all three conditions \eqref{eq:faces} hold, and since $F$ is $C^1$ they form a linear system for $\nabla F(0)$: $F_3=\alpha F_2$, $F_1=2F_2-F_3=(2-\alpha)F_2$, and $2F_1-F_2=-1$. Hence $(3-2\alpha)F_2(0)=-1$, that is, $e'(0)=F_2(0)=-(q+1)/(q+3)$. Also $e(0)=F(0,0,0)=\kappa u_n(0)$, which is positive by \eqref{eq:pairbounds}.

Next we derive an equation for $e$ from a single characteristic. The characteristic direction $(1,-1,1)$ preserves $y_1-y_3$, hence is tangent to the slice $\{y_1=y_3\}$, where it acts as $\partial_\sigma-\partial_\omega$. Since $\Psi_\sigma=F_1+F_3$ and $\Psi_\omega=F_2$, the face conditions \eqref{eq:faces} become
\begin{equation}\label{eq:psi}
(\partial_\sigma-\partial_\omega)^2\Psi=\Psi,\quad
\Psi_\omega=\frac12\Psi_\sigma\ \ (\omega=0),\quad
\Psi_\sigma=\beta\Psi_\omega-\frac12\ \ (\sigma=0),
\end{equation}
where $\beta=\frac12+\alpha=\frac{3q+1}{2(q+1)}$. For $q=2$, $\beta=7/6$ and \eqref{eq:psi} is~\cite[(3.17)]{calder2026explicit}. Write
\[
\mu=\beta-1=\frac{q-1}{2(q+1)}>0,\qquad d=\frac1\mu=\frac{2(q+1)}{q-1},\qquad
\gamma=d-1=\frac{q+3}{q-1},
\]
(for $q=2$: $\mu=\tfrac16$, $d=6$, $\gamma=5$), so that $e'(0)=-d/(2\gamma)$. Fix $X>0$. By \eqref{eq:fchar}, along the characteristic segment $s\mapsto(s,X-s)$, $0\leq s\leq X$, we have $\Psi(s,X-s)=e(X)\cosh s+D\sinh s$, where $D=(\Psi_\sigma-\Psi_\omega)(0,X)=\mu e'(X)-\tfrac12$ by the condition at $\sigma=0$. At $s=X$ this gives $H=e\cosh X+D\sinh X$, and since $\Psi_\omega=\tfrac12\Psi_\sigma$ at $\omega=0$, the derivative along the segment at $s=X$ is $\tfrac12H'$, so that $H'=2(e\sinh X+D\cosh X)$. The first relation gives
\[
e'(X)=d\Bigl(\frac{H(X)-e(X)\cosh X}{\sinh X}+\frac12\Bigr).
\]
Since $H$ and $e$ are $C^1$, the right side is $C^1$ on $X>0$, so $e\in C^2((0,\infty))$, and differentiating, with both relations used to eliminate $H$ and $H'$, gives
\begin{equation}\label{eq:eode}
e''+\gamma\coth X\,e'-d\,e=-\frac d2\coth X,
\end{equation}
which for $q=2$ is~\cite[(3.22)]{calder2026explicit}.

Finally, we expand $e'$ at the origin. Let $\eta=e-e''$ on $X>0$. By \eqref{eq:eode} and $\gamma e'(0)=-\tfrac d2$,
\begin{equation}\label{eq:eta}
\eta=\gamma\bigl(\coth X\,(e'-e'(0))-e\bigr),\qquad\text{and}\qquad
(\sinh^dX\,\eta)'=\frac d2\sinh^dX ,
\end{equation}
the second identity following from the first by one more use of \eqref{eq:eode}. The general solution of the second identity is $\eta=\tfrac d2\sinh^{-d}X\int_0^X\sinh^dt\,dt+C\sinh^{-d}X$. Since $e'$ is Lipschitz, the first identity shows that $\eta$ is bounded near $0$; therefore $C=0$ and
\begin{equation}\label{eq:etasol}
\eta(X)=\frac{d}{2\sinh^dX}\int_0^X\sinh^dt\,dt,\qquad0\leq\eta(X)\leq\frac d2\,X .
\end{equation}
The first identity in \eqref{eq:eta} can be rewritten as
\begin{equation}\label{eq:efirst}
e'(X)-e'(0)=\tanh X\Bigl(e(X)+\frac{\eta(X)}{\gamma}\Bigr),
\end{equation}
which for $q=2$ is~\cite[(2.9)]{calder2026explicit}. Since $e(X)=e(0)+O(X)$ and $\eta(X)=O(X)$, this is the expansion in \eqref{eq:e0}. No second-order expansion of $u_n$ has been assumed: beyond the $C^1$ regularity of $F$ used in the face conditions, the regularity of the true value enters only through the Lipschitz bound on $e'$, which excludes the singular solution of \eqref{eq:eta}. Finally, $e'$ is bounded and Lipschitz because $F$ has bounded and Lipschitz first derivatives.
\end{proof}

For $q=2$, \eqref{eq:etasol} gives $\eta(X)=3X/7+O(X^3)$, and \eqref{eq:e0} agrees with~\cite[(3.25)]{calder2026explicit}, with $e'(0)=-3/5$ and, by~\cite[Lemma~3.1]{calder2026explicit}, $e(0)=45\pi^2/512$.

\subsection{Transfer from two leading experts to three}\label{sec:transfer}
This subsection sets up a transport problem whose solution carries information from the two-leader ray, where Section~\ref{sec:scalar_trace} established the expansion of $e$, to the three-leader ray, where the contradiction will be found.

Define the traces $a(X,Y)=F(X,0,X+Y)$ and $\ell(X,Y)=F(0,X,Y)$, $X,Y\ge0$, as in~\cite[(3.7)]{calder2026explicit}. We use the operators of~\cite[(3.13)--(3.15)]{calder2026explicit},
\begin{equation}\label{eq:ops}
(\mathcal If)(X)=2\sinh^2X\int_X^\infty\frac{f(t)}{\sinh^3t}\,dt,\ \
\mathcal Kf=-\coth X\,f+\csch X\,\mathcal If,\ \ \mathcal T=\frac12\partial_X+\mathcal K,
\end{equation}
where $\mathcal T$ is the operator~\cite[(3.15)]{calder2026explicit} with $m=2$, whose solution operator $S_2(Y)$ is the propagator of~\cite[Lemma~3.3]{calder2026explicit}.

\begin{lemma}\label{lem:transport}
Suppose that $A$ satisfies \eqref{eq:fixed}. Then $g=\ell_Y$ is bounded and Lipschitz on $[0,\infty)^2$, it solves $g_Y=\mathcal Tg=\tfrac12g_X+\mathcal Kg$, with the operators \eqref{eq:ops}, in the integrated form along characteristics,
\begin{equation}\label{eq:mild}
g(X,Y)=g\bigl(X+\tfrac Y2,0\bigr)+\int_0^Y\bigl(\mathcal Kg(\cdot,Y')\bigr)\bigl(X+\tfrac{Y-Y'}2\bigr)\,dY'\qquad(X>0,\ Y\geq0),
\end{equation}
and
\begin{equation}\label{eq:gdata}
g(X,0)=F_3(0,X,0)=\alpha F_2(0,X,0)=\alpha e'(X).
\end{equation}
\end{lemma}

\begin{proof}
For $y_1\le y_3$ the characteristic through $y$ leaves the orthant at $(y_1+y_2,0,y_3+y_2)$ and at $(0,y_1+y_2,y_3-y_1)$, and the two-point problem for $\partial_s^2F=F$ gives~\cite[(3.8)]{calder2026explicit}: with $X=y_1+y_2$ and $Y=y_3-y_1$,
\begin{equation}\label{eq:interp}
F(y)=\frac{\sinh y_1}{\sinh X}\,a(X,Y)+\frac{\sinh y_2}{\sinh X}\,\ell(X,Y).
\end{equation}
The first two conditions of \eqref{eq:faces} are those of~\cite[(3.6)]{calder2026explicit}, so the derivation of~\cite[(3.10)]{calder2026explicit} applies unchanged and yields
\[
a_X=2\coth X\,a-2\csch X\,\ell,\qquad
\ell_X-2\ell_Y=2\coth X\,\ell-2\csch X\,a-1 .
\]
As in~\cite[\S3.2]{calder2026explicit}, the first equation integrates, since $(a/\sinh^2X)'=-2\ell/\sinh^3X$ and $a$ is bounded, to $a=\mathcal I\ell$, and the second then reads $\ell_Y=\mathcal T\ell+\tfrac12$, pointwise on $X>0$, $Y\geq0$, because $\ell$ is $C^1$. Integrating along the characteristics $s\mapsto(X+\tfrac{Y-s}2,s)$ gives the integrated form of $\ell_Y=\mathcal T\ell+\tfrac12$, and the same holds for $\ell(\cdot,\cdot+\delta)$, $\delta>0$. Hence the difference quotients $Q_\delta=(\ell(\cdot,\cdot+\delta)-\ell)/\delta$ satisfy \eqref{eq:mild} exactly, the constant forcing having dropped out. The function $g=\ell_Y$ is bounded and Lipschitz, since $F$ is $C^{1,1}$ with bounded first derivatives, so $Q_\delta\to g$ uniformly as $\delta\downarrow0$. On $X\geq X_0>0$ the operator $\mathcal K$ is bounded in the supremum norm and uses only values at $t\geq X$, and the characteristics stay in $X\geq X_0$, so \eqref{eq:mild} passes to the limit. Finally, the third condition in \eqref{eq:faces} gives \eqref{eq:gdata}.
\end{proof}

Observe that nothing in Lemma~\ref{lem:transport} or its proof depends on $q$ except the scalar $\alpha$; together with $\beta$, which fixed $e$, these two numbers carry the entire $q$-dependence. This is why a single computation covers every $n\ge5$, with five experts appearing as the boundary case $q=2$.

Lemma~\ref{lem:kernel} shows that if we start the transport equation $g_Y = \mathcal Tg$ from data $\psi$ on the axis $Y = 0$, we can deduce the solution on the corner $X = 0$ for small $Y > 0$, up to first order terms in $Y$. Its main idea is that this corner value is given by an exact integral kernel applied to the data, and that kernel behaves like an approximate identity whose first moment is $3Y/2$.

\begin{lemma}\label{lem:kernel}
Let $\psi$ be bounded and continuous on $[0,\infty)$, and let $g$ be a bounded continuous function on $[0,\infty)^2$ with $g(X,0)=\psi(X)$ that solves \eqref{eq:mild}. Then its trace at $X=0$ is given, with $r=Y/2>0$, by
\begin{equation}\label{eq:kernel}
g(0,Y)=6\sinh^2r\int_r^\infty\frac{\sinh(t-r)}{\sinh^4t}\,\psi(t)\,dt .
\end{equation}
If $\psi(t)=\psi_0+\psi_1t+O(t^2)$ at zero, then $g(0,Y)=\psi_0+\tfrac32\psi_1Y+O(Y^2\log(1/Y))$.
\end{lemma}

\begin{proof}
We first construct the solution and its kernel. Put $L=X+Y/2$, $r=Y/2$, and let $\Phi_t(s)=\sinh(t-s)/\sinh t$ for $0\le s\le t$, extended by zero for $s>t$, be the truncated mode~\cite[(2.13)]{calder2026explicit}. Define, as in~\cite[(3.31)--(3.32)]{calder2026explicit} with $m=2$,
\begin{equation}\label{eq:semigroup}
(S_2(Y)\psi)(X)=\Phi_L(r)^2\psi(L)
+\int_L^\infty\frac{2}{\sinh^3t}\bigl[2\sinh L\sinh r\,\Phi_t(r)
+\sinh^2r\,\Phi_t(L)\bigr]\psi(t)\,dt .
\end{equation}
For $t>L$ the density separates as $p_1\sinh^{-3}t+p_2\cosh t\sinh^{-4}t$, with coefficients $p_1=2\sinh r\,(2\sinh L\cosh r+\sinh r\cosh L)$ and $p_2=-6\sinh L\sinh^2r$ that depend only on $L$ and $r$. Hence
\[
(S_2(Y)\psi)(X)=\Phi_L(r)^2\psi(L)+p_1J_3(L)+p_2J_4(L),
\]
where $J_3(L)=\int_L^\infty\psi(t)\sinh^{-3}t\,dt$ and $J_4(L)=\int_L^\infty\psi(t)\cosh t\,\sinh^{-4}t\,dt$, and $S_2(Y)\psi$ is bounded on $X\ge\delta$ for every $\delta>0$ on each finite $Y$-interval. We verify directly that it solves \eqref{eq:mild}. Write $X=L-r$, $p_3=2\sinh L\,(\sinh L\cosh r+2\cosh L\sinh r)$ and $p_4=-6\sinh^2L\sinh r$. The identities
\[
\begin{gathered}
(\partial_Lp_3)\sinh X-2p_3\cosh X=-2p_1,\qquad
(\partial_Lp_4)\sinh X-2p_4\cosh X=-2p_2,\\
p_3\sinh L+p_4\cosh L=2\sinh X\sinh^2L
\end{gathered}
\]
show that $(p_3J_3+p_4J_4)/\sinh^2X$ has derivative $-2S_2(Y)\psi/\sinh^3X$ in $X$ at fixed $Y$, and it tends to zero as $X\to\infty$, so that $\mathcal I[S_2(Y)\psi]=p_3J_3(L)+p_4J_4(L)$. Along a characteristic $L$ is constant and $r$ increases at rate $\tfrac12$, and the identities
\[
\begin{gathered}
\partial_r\bigl(\Phi_L(r)^2\bigr)\sinh X=-2\Phi_L(r)^2\cosh X,\\
(\partial_rp_1)\sinh X=2(p_3-p_1\cosh X),\qquad
(\partial_rp_2)\sinh X=2(p_4-p_2\cosh X)
\end{gathered}
\]
give $\tfrac12\partial_rS_2(Y)\psi=-\coth X\,S_2(Y)\psi+\csch X\,\mathcal I[S_2(Y)\psi]=\mathcal K[S_2(Y)\psi]$. All six identities follow from the addition formulas for $\sinh(L-r)$ and $\cosh(L-r)$. Since $S_2(0)\psi=\psi$, the function $S_2(Y)\psi$ solves \eqref{eq:mild}. No approximation of the data is needed, since $\psi$ enters only through $\psi(L)$, $J_3$ and $J_4$; the atom is $G'(L)$ and the density $G''-G$ for $G(t)=\sinh t\,\Phi_t(L)\Phi_t(r)^2$, which is how~\cite{calder2026explicit} obtains \eqref{eq:semigroup} from its truncated modes.

Next we prove uniqueness. Fix $\delta>0$. On $X\ge\delta$ the operator $\mathcal K$ is bounded, $\|\mathcal K\|\le2\coth\delta$ (from $\int_X^\infty\csch^3t\,dt\le\tfrac12\csch^2X$), and uses only values at $t\ge X$. The characteristics in \eqref{eq:mild} stay in $X\ge\delta$, so Gronwall's inequality, applied to the difference of two bounded solutions of \eqref{eq:mild} with the same data, shows that they agree on $X\ge\delta$. Hence the given solution equals $S_2(Y)\psi$ for all $X>0$.

We now compute the corner trace. Let $X\downarrow0$ in \eqref{eq:semigroup}: $L\to r$, the atom vanishes since $\Phi_r(r)=0$, and the bracket collapses to $3\sinh^2r\,\Phi_t(r)$, giving \eqref{eq:kernel}. The limit is the trace $g(0,Y)$ because $g$ is continuous.

Finally, we expand the corner trace. Write $K_r(t)$ for the density in \eqref{eq:kernel}. On $r\le t\le1$, the bounds $x\leq\sinh x\leq x+x^3/4$ for $0\leq x\leq1$ give $K_r(t)=6r^2(t-r)t^{-4}(1+O(t^2))$, and the rescaled density $6(x-1)x^{-4}$ on $[1,\infty)$ has mass $1$ and first moment $3$; beyond $t=1$ the density is $O(r^2e^{-3t})$. Hence $\int_r^\infty K_r=1+O(r^2\log(1/r))$, $\int tK_r=3r+O(r^2)$ and $\int t^2K_r=O(r^2\log(1/r))$, and inserting the expansion of $\psi$ with $r=Y/2$ proves the claim.
\end{proof}

We can now compute the expansion of $u_n$ along the three-leader ray.

\begin{proposition}\label{prop:ray}
Suppose that $A$ satisfies \eqref{eq:fixed}. Then $f_3(s)=u_n(s,s,s,0,\dots,0)$ satisfies, as $s\downarrow0$,
\begin{equation}\label{eq:f3}
f_3(s)=u_n(0)+\frac3n\,s+\frac{3q}{2(q+1)}\,u_n(0)\,s^2+o(s^2).
\end{equation}
\end{proposition}

\begin{proof}
By Lemma~\ref{lem:transport}, we may apply Lemma~\ref{lem:kernel} to $g=\ell_Y$ with $\psi=\alpha e'$, which is bounded and continuous on $[0,\infty)$ by Lemma~\ref{lem:trace}, with $\psi_0=\alpha e'(0)$ and $\psi_1=\alpha e(0)$ by \eqref{eq:e0}. Integrating the resulting expansion of $\ell_Y(0,Y)$ in $Y$, legitimate because $\ell$ is $C^1$,
\[
\ell(0,Y)=e(0)+\alpha e'(0)\,Y+\frac34\alpha e(0)\,Y^2+O(Y^3\log(1/Y)).
\]
On the three-leader ray we have $(y_1,y_2,y_3)=(0,0,\kappa s)$, so $f_3(s)=s+\ell(0,\kappa s)/\kappa$. With $e(0)=\kappa u_n(0)$, $\kappa^2=2$ and $1+\alpha e'(0)=3/(q+3)=3/n$, this is \eqref{eq:f3}.
\end{proof}

For $n=6$ the quadratic coefficient is $9u_6(0)/8$. The linear coefficient $3/n$ is also the value forced independently by permutation symmetry and $\nabla u_n\cdot\one=1$, which checks \eqref{eq:faces} and the normalization of \eqref{eq:eode}.

\begin{remark}[What a Hessian at the origin would give]\label{rem:hessian}
If $u_n$ were twice differentiable at the origin, symmetry and $D^2u_n(0)\one=0$ would force $D^2u_n(0)=\lambda_n(I-\tfrac1n\one\one^T)$, hence $D_{\v}^2u_n(0)=\lambda_n(k-k^2/n)$ for a direction with $k$ ones. Normalizing by the equation for $A$ ($k=2$) gives $\lambda_n=nu_n(0)/(n-2)$ and, for $C=(1,1,1,0,\dots,0)$, $\tfrac12f_3''(0)=\tfrac{3q}{2(q+1)}u_n(0)$, exactly the coefficient in \eqref{eq:f3}. Since $3-9/n>2-4/n$ precisely when $n\ge6$, a three-one direction would beat $A$ at the origin. The content of Sections~\ref{sec:scalar_trace} and~\ref{sec:transfer} is that the same coefficient follows from the trace equations and the $C^{1,1}$ bound alone, with no Hessian at the origin or on either collision ray.
\end{remark}

\subsection{The contradiction on the three-leader ray}\label{sec:contradiction}

We now compare \eqref{eq:f3} with the Bellman inequality for the direction of the ray, which proves Theorem~\ref{thm:nonexistence}.

\begin{proof}[Proof of Theorem~\ref{thm:nonexistence}]
Let $n\geq6$. Complementing if necessary, which is legitimate by the discussion after Lemma~\ref{lem:lines}, we may assume that the first bit is one. By Theorem~\ref{thm:exclude}, only $A=1010\cdots0$ can then satisfy \eqref{eq:fixed}, so suppose, for contradiction, that it does. Then Proposition~\ref{prop:ray} gives \eqref{eq:f3}. The direction $C=(1,1,1,0,\dots,0)$ is the direction of the ray, and $u_n$ is a viscosity supersolution of $u_n-\tfrac12D_C^2u_n\geq\phi_n$, so Lemma~\ref{lem:lines}(iii), applied along the ray with $h=s$, gives
\begin{equation}\label{eq:bellman-ray}
f_3(2s)+f_3(0)-2f_3(s)\le2s^2\max_{0\leq t\leq2s}\bigl(f_3(t)-t\bigr)=2s^2\bigl(u_n(0)+O(s)\bigr),
\end{equation}
because $\phi_n=t$ on the ray and $u_n-\phi_n$ is Lipschitz. By \eqref{eq:f3}, the left side is $\tfrac{3q}{q+1}u_n(0)s^2+o(s^2)$, so subtracting leaves
\[
\Bigl(\frac{3q}{q+1}-2\Bigr)u_n(0)\,s^2=\frac{q-2}{q+1}\,u_n(0)\,s^2\le o(s^2),
\]
which is impossible for small $s>0$, because $q\ge3$ and $u_n(0)>0$. This excludes $1010\cdots0$ for every $n\ge6$.
\end{proof}

\begin{remark}[The boundary case $q=2$]\label{rem:q2}
For five experts the coefficient in \eqref{eq:f3} is $\tfrac{3\cdot2}{2\cdot3}=1$, so \eqref{eq:bellman-ray} is saturated at order $s^2$ rather than violated, consistent with the optimality of $10100$ in~\cite{calder2026explicit}. Indeed, the explicit expansion~\cite[(2.18)]{calder2026explicit} on the ray $(s,s,s,0,0)$ gives $f_3(s)=u_5(0)+\tfrac35s+u_5(0)s^2+o(s^2)$, matching \eqref{eq:f3} in both coefficients. Since that expansion is derived from the closed-form five-expert solution and not from the trace equations, it is an independent check of the chain \eqref{eq:eode}--\eqref{eq:f3}.
\end{remark}

Everything above uses only the $C^{1,1}$ regularity of the true Bellman value from Section~\ref{sec:basic}, together with the explicit form of \eqref{eq:fixed} on segments and the second-difference inequality \eqref{eq:seconddiff} of Lemma~\ref{lem:lines}. The derivative of the trace $e$ acquired a first-order expansion through its scalar ODE, whose singular solution in \eqref{eq:eta} is excluded by the Lipschitz bound; the transfer kernel then produced the quadratic expansion of a different trace. No Hessian of $u_n$ at the origin or on either collision ray was assumed, so the argument does not settle whether $u_n$ is $C^2$.
\section{The finite-horizon problem}\label{sec:finite}
This section proves Theorem~\ref{thm:finite}. The geometric nonexistence theorem implies its finite-horizon analogue by a short transfer argument. There is no need to rederive the one-leader asymptotic or the trace contradiction of Section~\ref{sec:last} for the parabolic equation. The transfer is applied to the linear equations attached to a hypothetical common optimizer, and therefore does not assert that a Laplace transform commutes with the nonlinear Bellman Hamiltonian.

\subsection{The finite-horizon value}
Let $U_n(\tau,x)$ be the value of \eqref{eq:parabolic}, with time to maturity $\tau$. We use its control representation
\begin{equation}\label{eq:parcontrol}
 U_n(\tau,x)=\sup_\alpha\E\,\phi_n\!\left(
   x+\int_0^\tau\alpha_s\,dW_s\right),
\end{equation}
the supremum being over progressively measurable $\alpha$ with values in $\{0,1\}^n$ and $W$ a scalar Brownian motion, as in Remark~\ref{rem:finitegeometric}. We write $X^\alpha_\tau=x+\int_0^\tau\alpha_s\,dW_s$ for the controlled state. The diffusion matrix $\alpha\alpha^T$ is exactly the rank one matrix $\v\v^T$ appearing in the Hamiltonian.

\begin{lemma}\label{lem:pargrowth}
For every $\tau\geq0$, the function $U_n(\tau,\cdot)$ is convex, permutation invariant, translation equivariant, and $1$-Lipschitz for the supremum norm, and $U_n$ is nondecreasing in $\tau$. Moreover, for all $\tau\geq0$ and $x\in\R^n$,
\begin{equation}\label{eq:pargrowth}
 0\leq U_n(\tau,x)-\phi_n(x)\leq n\sqrt\tau ,
\end{equation}
and for every $\lambda>0$,
\begin{equation}\label{eq:parscaling}
 U_n(\lambda^2\tau,\lambda x)=\lambda\,U_n(\tau,x),
 \qquad\text{equivalently}\qquad
 U_n(\tau,x)=\sqrt\tau\,\Theta\!\left(x/\sqrt\tau\right),
 \quad \Theta:=U_n(1,\cdot).
\end{equation}
\end{lemma}
\begin{proof}
The first statements follow from \eqref{eq:parcontrol}: each $x\mapsto\E\,\phi_n(X^\alpha_\tau)$ is convex, translation equivariant and $1$-Lipschitz because $\phi_n$ is, permuting coordinates maps admissible controls to admissible controls, and $U_n$ is nondecreasing in $\tau$ because a control may be extended by $\alpha\equiv0$. The lower bound in \eqref{eq:pargrowth} is this monotonicity. For the upper bound, write $M_\tau^i=\int_0^\tau\alpha^i_s\,dW_s$. Each $M^i$ is a martingale with $\langle M^i\rangle_\tau\leq\tau$, so $\E|M^i_\tau|\leq\sqrt\tau$. Since $\phi_n$ is $1$-Lipschitz for the supremum norm,
\[
 \E\phi_n(x+M_\tau)-\phi_n(x)\leq\E\max_i|M^i_\tau|
   \leq\sum_{i=1}^n\E|M^i_\tau|\leq n\sqrt\tau ,
\]
uniformly in $x$ and in $\alpha$. For \eqref{eq:parscaling}, substitute $s=\lambda^2r$ in \eqref{eq:parcontrol} and use that $\widetilde W_r=\lambda^{-1}W_{\lambda^2r}$ is a Brownian motion, so that
\[
 \int_0^{\lambda^2\tau}\alpha_s\,dW_s
  =\lambda\int_0^{\tau}\alpha_{\lambda^2r}\,d\widetilde W_r .
\]
The correspondence $\alpha\mapsto\alpha_{\lambda^2\cdot}$ is a bijection of admissible controls, so, by the degree one homogeneity of $\phi_n$,
\[
 U_n(\lambda^2\tau,\lambda x)
  =\sup_{\widetilde\alpha}\E\,\phi_n\!\left(\lambda\Bigl(
      x+\int_0^\tau\widetilde\alpha_r\,d\widetilde W_r\Bigr)\right)
  =\lambda\,U_n(\tau,x).
\]
Taking $\lambda=1/\sqrt\tau$ gives the self-similar form.
\end{proof}

We next record the parabolic counterpart of \eqref{eq:linearineq}. As in Section~\ref{sec:basic}, no equivalence theorem is needed: the dynamic programming principle gives the distributional inequality directly.

\begin{lemma}\label{lem:parineq}
For every $\v'\in\{0,1\}^n$,
\begin{equation}\label{eq:parlinearineq}
 \partial_\tau U_n-\frac12D_{\v'}^2U_n\geq0
 \qquad\text{in }\mathcal D'\bigl((0,\infty)\times\R^n\bigr).
\end{equation}
\end{lemma}
\begin{proof}
Playing the constant control $\v'$ on a first time slice of length $h>0$ and then playing optimally  gives the dynamic programming inequality
\[
 U_n(\tau+h,x)\geq\E\,U_n(\tau,x+\v'W_h).
\]
Let $0\leq\psi\in C_c^\infty((0,\infty)\times\R^n)$ and write $\psi_\tau=\psi(\tau,\cdot)$. Translating the variable of integration,
\[
 \int_{\R^n}\E\,U_n(\tau,x+\v'W_h)\psi_\tau(x)\,dx
  =\int_{\R^n}U_n(\tau,y)\,\E\psi_\tau(y-\v'W_h)\,dy .
\]
Taylor expansion to third order, together with $\E W_h=0$, $\E W_h^2=h$ and $\E|W_h|^3=O(h^{3/2})$, gives $\E\psi_\tau(\cdot-\v'W_h)=\psi_\tau+\tfrac h2D_{\v'}^2\psi_\tau+R_{h,\tau}$ with
\[
 \int_{\R^n}(1+|y|)\,|R_{h,\tau}(y)|\,dy=O(h^{3/2})
\]
uniformly in $\tau$. Indeed, on the event $\{|W_h|\leq1\}$ the third-order Taylor remainder is $O(|W_h|^3)$ and supported in a fixed compact set, while on the complementary event, whose probability is at most $2e^{-1/(2h)}$, each term of the expansion has weighted norm $\int(1+|y|)|\cdot|\,dy$ at most $C(1+W_h^2)$. Since $|U_n(\tau,y)|\leq C(1+|y|)$ on the time support of $\psi$ by \eqref{eq:pargrowth}, integrating the dynamic programming inequality against $\psi$ gives
\[
 \frac1h\iint\bigl[U_n(\tau+h,x)-U_n(\tau,x)\bigr]\psi(\tau,x)\,dx\,d\tau
 \geq\frac12\iint U_n\,D_{\v'}^2\psi\,dy\,d\tau+O(h^{1/2}).
\]
Shifting the time variable in the first term, with $\psi$ extended by zero to $\tau\leq0$, the left side equals
\[
 \iint U_n(\tau,x)\,\frac{\psi(\tau-h,x)-\psi(\tau,x)}{h}\,dx\,d\tau ,
\]
which tends to $-\iint U_n\,\partial_\tau\psi$ as $h\downarrow0$ by dominated convergence and \eqref{eq:pargrowth}. This is \eqref{eq:parlinearineq}.
\end{proof}

We read \eqref{eq:parfixed} in the sense of distributions on $(0,\infty)\times\{x_1>\cdots>x_n\}$. For the constant-coefficient linear operator $\partial_\tau-\tfrac12D_{\v}^2$ this is equivalent to the viscosity formulation~\cite{ishii1995equivalence}, but the proofs below use only the distributional form.

\subsection{A common finite-horizon optimizer is a geometric optimizer}

The key to passing the result to the finite horizon problem is the following lemma, which shows that a single rank direction satisfying the finite-horizon equation \eqref{eq:parfixed} at every time to maturity also satisfies the geometric equation \eqref{eq:fixed}. 

\begin{lemma}\label{lem:transferparabolic}
Suppose $\v\in\{0,1\}^n$ satisfies \eqref{eq:parfixed}. Then
\begin{equation}\label{eq:laplace}
 \mathcal G_n(x):=\int_0^\infty e^{-\tau}U_n(\tau,x)\,d\tau
\end{equation}
is the geometric Bellman value $u_n$ of \eqref{eq:pde}, and $\v$ satisfies \eqref{eq:fixed}. No classical second derivatives of $U_n$ at collisions are assumed.
\end{lemma}
\begin{proof}
We first pass the inequalities \eqref{eq:parlinearineq} and the equation \eqref{eq:parfixed} to the transform. By \eqref{eq:pargrowth} the integral \eqref{eq:laplace} converges absolutely, and
\begin{equation}\label{eq:gbound}
 0\leq\mathcal G_n-\phi_n\leq n\int_0^\infty e^{-\tau}\sqrt\tau\,d\tau
   =n\,\Gamma\!\left(\frac32\right)=\frac{n\sqrt\pi}{2}.
\end{equation}
Let $S$ be a smooth nondecreasing function with $S=0$ on $(-\infty,0]$ and $S=1$ on $[1,\infty)$, and for $N\geq2$ put $\chi_N(\tau)=e^{-\tau}\bigl(S(N\tau-1)-S(\tau-N)\bigr)$. Then $\chi_N\in C_c^\infty(0,\infty)$ is supported in $[1/N,N+1]$, $\chi_N(\tau)\to e^{-\tau}$ for every $\tau>0$, and
\[
 -\chi_N'(\tau)=\chi_N(\tau)-Ne^{-\tau}S'(N\tau-1)+e^{-\tau}S'(\tau-N).
\]
Test \eqref{eq:parlinearineq} against $\chi_N(\tau)\psi(x)$ with $0\leq\psi\in C_c^\infty(\R^n)$, and the equality \eqref{eq:parfixed} against $\chi_N(\tau)\psi(x)$ with $\psi\in C_c^\infty(\{x_1>\cdots>x_n\})$. In the $\partial_\tau$ term the three pieces of $-\chi_N'$ contribute as follows. The first converges to $\int_0^\infty e^{-\tau}U_n\,d\tau=\mathcal G_n$ by dominated convergence, since $U_n$ may be negative and $e^{-\tau}(|\phi_n(x)|+n\sqrt\tau)$ dominates it by \eqref{eq:pargrowth}. After the substitution $s=N\tau-1$, the second is $-\int_0^1S'(s)e^{-\tau}U_n(\tau,x)\,ds$ with $\tau=(s+1)/N\downarrow0$; since $\int_0^1S'=1$ and $U_n(\tau,x)\to\phi_n(x)$ as $\tau\downarrow0$ by \eqref{eq:pargrowth}, it converges to $-U_n(0,x)=-\phi_n(x)$ by dominated convergence. The third vanishes for each fixed $\tau$ once $N>\tau$, so it tends to zero by dominated convergence. In the second-order term, dominated convergence gives $\tfrac12\int\mathcal G_nD_{\v'}^2\psi$. Therefore
\begin{equation}\label{eq:gineq}
 \mathcal G_n-\frac12D_{\v'}^2\mathcal G_n\geq\phi_n
 \ \text{ in }\mathcal D'(\R^n)\ \text{ for every }\v',
 \qquad
 \mathcal G_n-\frac12D_{\v}^2\mathcal G_n=x_1
 \ \text{ in }\mathcal D'\bigl(\{x_1>\cdots>x_n\}\bigr).
\end{equation}

Next we prove regularity and the supersolution property. The function $\mathcal G_n$ is convex, permutation invariant and translation equivariant, being an average of functions with these properties. The first part of \eqref{eq:gineq} is \eqref{eq:linearineq} for $\mathcal G_n$, so the proof of Proposition~\ref{prop:hessian} gives $0\leq\Delta\mathcal G_n\leq8(\mathcal G_n-\phi_n)$, and hence $\mathcal G_n\in C^{1,1}(\R^n)$, the right side being bounded by \eqref{eq:gbound}. For the supersolution property, let $\rho_k\geq0$ be a smooth even mollifier supported in the ball of radius $1/k$, and put $\mathcal G_n^{(k)}=\mathcal G_n*\rho_k$. Testing the first part of \eqref{eq:gineq} against $\rho_k(x-\cdot)$ gives $D_{\v'}^2\mathcal G_n^{(k)}\leq2\bigl((\mathcal G_n-\phi_n)*\rho_k\bigr)$ pointwise, so
\[
 \mathcal G_n^{(k)}(x+h\v')+\mathcal G_n^{(k)}(x-h\v')-2\mathcal G_n^{(k)}(x)\leq2h^2\max_{|s|\leq h}\bigl((\mathcal G_n-\phi_n)*\rho_k\bigr)(x+s\v'),
\]
and letting $k\to\infty$ gives the second-difference inequality \eqref{eq:seconddiff} for $\mathcal G_n$ on every line in the direction $\v'$. If $\psi\in C^2$ touches $\mathcal G_n$ from below at $x_0$, the second differences of $\psi$ at $x_0$ are at most those of $\mathcal G_n$, so dividing by $h^2$ and letting $h\downarrow0$ gives $\mathcal G_n(x_0)-\tfrac12D_{\v'}^2\psi(x_0)\geq\phi_n(x_0)$ for every $\v'$, and the maximum over the finitely many $\v'$ may be taken inside.

Next we show that $\mathcal G_n$ satisfies \eqref{eq:fixed}. At points $y$ whose distance to the boundary of the sector exceeds $1/k$, testing the second part of \eqref{eq:gineq} against $\rho_k(y-\cdot)$ gives $\mathcal G_n^{(k)}(y)-\tfrac12D_{\v}^2\mathcal G_n^{(k)}(y)=y_1$, because $\rho_k$ is even. On a segment of the sector in the direction $\v$, the function $t\mapsto\mathcal G_n^{(k)}(x+t\v)$ is therefore, for large $k$, a classical solution of $g-\tfrac12g''=x_1+v_1t$, so that $\mathcal G_n^{(k)}(x+t\v)=x_1+v_1t+A_ke^{\kappa t}+B_ke^{-\kappa t}$. Since $\mathcal G_n^{(k)}\to\mathcal G_n$ locally uniformly, and a function of this form on a segment is determined by two of its values, $\mathcal G_n(x+t\v)$ has the same form. By Lemma~\ref{lem:lines}(v), $\mathcal G_n$ satisfies \eqref{eq:fixed} in the viscosity sense.

We now prove the subsolution property. On each open sector $\{x_{\sigma(1)}>\cdots>x_{\sigma(n)}\}$, permutation invariance transforms \eqref{eq:fixed} into the linear equation for the correspondingly permuted direction $\hat\v$, so if $\psi\in C^2$ touches $\mathcal G_n$ from above at a point of that sector, then $\mathcal G_n-\tfrac12\max_{\v'}D_{\v'}^2\psi\leq\mathcal G_n-\tfrac12D_{\hat\v}^2\psi\leq\phi_n$ there. It remains to treat collisions; this step uses only that $\mathcal G_n\in C^{1,1}$, and no second derivatives of $\mathcal G_n$. Suppose that $\psi\in C^2$, that $\mathcal G_n-\psi$ has a local maximum at $x_0$, and, for contradiction, that
\[
 \mathcal G_n(x_0)-\frac12\max_{\v'}D_{\v'}^2\psi(x_0)>\phi_n(x_0)+2\eta
\]
for some $\eta>0$. Fix $\delta>0$ with $\delta n<\eta$, let $w(x)=\mathcal G_n(x)-\psi(x)-\delta|x-x_0|^2$, and choose $r>0$ such that $\psi$ is $C^2$ on $\overline B_r(x_0)$, $w(x)\leq w(x_0)-\delta|x-x_0|^2$ there, and $\mathcal G_n-\tfrac12\max_{\v'}D_{\v'}^2\psi>\phi_n+\eta$ on $\overline B_r(x_0)$. For $|p|<\delta r$, every maximum point $y$ of $w(x)+p\cdot x$ over $\overline B_r(x_0)$ satisfies $\delta|y-x_0|^2\leq p\cdot(y-x_0)$, hence $|y-x_0|<r$, so $y$ is an interior maximum and $\nabla w(y)=-p$. The set $K$ of these maximum points, over all $|p|<\delta r$, therefore satisfies $\nabla w(K)\supseteq B_{\delta r}(0)$. Since $\nabla w$ is Lipschitz on $\overline B_r(x_0)$, and Lipschitz maps send sets of measure zero to sets of measure zero, $K$ is not contained in the collision set $\bigcup_{i\neq j}\{x_i=x_j\}$, which has measure zero. At a point $y\in K$ off the collision set, the open-sector case applied to the test function $\psi+\delta|x-x_0|^2-p\cdot x$ gives
\[
 \mathcal G_n(y)-\frac12\max_{\v'}D_{\v'}^2\psi(y)-\delta n\leq\phi_n(y),
\]
because $D_{\v'}^2\bigl(\delta|x-x_0|^2\bigr)=2\delta|\v'|^2\leq2\delta n$ and the linear term has no second derivatives. Since $\delta n<\eta$, this contradicts the choice of $r$. Hence $\mathcal G_n$ is a viscosity solution of \eqref{eq:pde}.

Finally, we identify $\mathcal G_n$ with $u_n$. By \eqref{eq:gbound}, $\mathcal G_n$ has at most linear growth, so Lemma~\ref{lem:comparison}, applied in both directions, gives $\mathcal G_n=u_n$, and $\v$ satisfies \eqref{eq:fixed}.
\end{proof}

\begin{proof}[Proof of Theorem~\ref{thm:finite}]
If $\v$ satisfied \eqref{eq:parfixed} on $(0,\infty)\times\{x_1>\cdots>x_n\}$, Lemma~\ref{lem:transferparabolic} would make $\v$ satisfy \eqref{eq:fixed}, contradicting Theorem~\ref{thm:nonexistence} for $n\geq6$. The remaining assertions follow from Remark~\ref{rem:selfsimilar}.
\end{proof}

\begin{remark}[Self-similarity, and why one time suffices]\label{rem:selfsimilar}
By \eqref{eq:parscaling} the value is determined by the single profile $\Theta=U_n(1,\cdot)$, which is convex and Lipschitz, through $U_n(\tau,\sqrt\tau\,y)=\sqrt\tau\,\Theta(y)$ for $\tau>0$. For $\psi\in C_c^\infty((0,\infty)\times\R^n)$ put $\zeta_\tau(y)=\psi(\tau,\sqrt\tau\,y)$, so that $\partial_\tau[\zeta_\tau(y)]=\partial_\tau\psi+(2\sqrt\tau)^{-1}y\cdot\nabla_x\psi$, $\nabla\zeta_\tau=\sqrt\tau\,\nabla_x\psi$ and $D_{\v}^2\zeta_\tau=\tau D_{\v}^2\psi$, all evaluated at $(\tau,\sqrt\tau\,y)$. Substituting $x=\sqrt\tau\,y$ gives the exact identity
\begin{equation}\label{eq:selfsimilar}
 \iint U_n\Bigl(\partial_\tau\psi+\frac12D_{\v}^2\psi\Bigr)dx\,d\tau
 =-\frac12\int_0^\infty\tau^{(n-1)/2}\int\Theta(y)\bigl((n+1)\zeta_\tau+y\cdot\nabla\zeta_\tau-D_{\v}^2\zeta_\tau\bigr)dy\,d\tau ,
\end{equation}
because the remaining term $\Theta(y)\,\partial_\tau\bigl(\tau^{(n+1)/2}\zeta_\tau(y)\bigr)$ integrates to zero in $\tau$ for each fixed $y$. The inner integral is the $\tau$-free identity
\[
 \Theta-y\cdot D\Theta-D_{\v}^2\Theta=0
\]
tested against $\zeta_\tau$. Hence, by \eqref{eq:selfsimilar}, \eqref{eq:parfixed} is equivalent to the $\tau$-free identity in distributions on the sector. Indeed, the sector is a cone, so each $\zeta_\tau$ is a test function on the sector whenever $\psi$ is a test function on $(0,\infty)$ times the sector, and conversely, testing \eqref{eq:parfixed} with $\psi(\tau,x)=\chi(\tau)\zeta(x/\sqrt\tau)$, for which $\zeta_\tau=\chi(\tau)\zeta$, recovers the $\tau$-free identity whenever $0\leq\chi\in C_c^\infty(0,\infty)$ is not identically zero, since then $\int\tau^{(n-1)/2}\chi\,d\tau>0$. Since $U_n$ is only known to be Lipschitz, we say that a rank direction $\v$ attains the Hamiltonian maximum throughout the spatial sector at a single time $\tau_0>0$ if $\Theta$ satisfies this identity; by \eqref{eq:parscaling}, this is the identity $U_n-x\cdot DU_n-\tau_0D_{\v}^2U_n=0$ at $\tau=\tau_0$, which is $\partial_\tau U_n=\tfrac12D_{\v}^2U_n$ at $\tau=\tau_0$ written through the scaling. This notion is unambiguous because $\Theta\in C^{1,1}_{\mathrm{loc}}$. Indeed, averaging \eqref{eq:parlinearineq} over $\v'$ uniform on $\{0,1\}^n$, as in Section~\ref{sec:basic}, and using the convexity of $\Theta$ gives $0\leq\Delta\Theta\leq4(\Theta-y\cdot D\Theta)$ in distributions, and the right side is locally bounded, so the Hessian-measure argument of Section~\ref{sec:basic} applies. Hence the identity holds in distributions if and only if it holds almost everywhere, with the classical second derivatives of $\Theta$. Three consequences follow. First, an identity on any nonempty open time interval extends to all positive times, because the converse argument above only uses $\chi$ supported in that interval. Second, a direction that attains the maximum throughout the sector at a single time $\tau_0>0$ satisfies the $\tau$-free identity, and hence \eqref{eq:parfixed} at every positive time. Third, a direction that depends on the remaining time is, at any single time, a rank direction of this kind, so allowing such dependence cannot help.
\end{remark}

\begin{remark}[Control interpretation]\label{rem:parcontrol}
There is a direct control counterpart of Lemma~\ref{lem:transferparabolic}, which avoids interchanging a supremum with a time integral. If a single rank policy $\v$ attained the value \eqref{eq:parcontrol} for every horizon, then $\E\phi_n(X^\alpha_\tau)\leq\E\phi_n(X^v_\tau)$ for every admissible $\alpha$ and every $\tau>0$. Integrating this inequality against $e^{-\tau}d\tau$, which is legitimate by \eqref{eq:pargrowth}, makes $\v$ optimal for the geometric objective $\E\phi_n(X_T)$ with $T$ an independent unit exponential. This is a statement about policies rather than about the equation, and turning it into \eqref{eq:fixed} would require a verification theorem; the proof of Lemma~\ref{lem:transferparabolic} works directly with the equations instead.
\end{remark}

\begin{remark}[Relation to the four-expert transfer]\label{rem:bez}
This direction of transfer is the principle used for COMB by Bayraktar, Ekren and Zhang~\cite[Proposition 3.1]{bayraktar2020finite}, whose identity $u(x)=\int_0^\infty e^{-\tau}U_n(\tau,x)\,d\tau$ is \eqref{eq:laplace}. Their statement assumes that the finite-horizon value is $C^{1,2}$ and satisfies an integrability condition on its spatial Hessian. Lemma~\ref{lem:transferparabolic} removes both hypotheses: it uses only the dynamic programming inequalities of Lemma~\ref{lem:parineq}, the growth bound \eqref{eq:pargrowth}, Lemma~\ref{lem:lines}, the $C^{1,1}$ argument of Section~\ref{sec:basic} and the comparison principle of Lemma~\ref{lem:comparison}, and it produces the viscosity form of \eqref{eq:fixed} rather than a pointwise one. Kobzar, Kohn and Wang~\cite{kobzar2020newgeometric} use the same transform in time to build sub- and supersolutions of the geometric problem from finite-horizon ones, which gives explicit bounds on the value.

The transform \eqref{eq:laplace} also reproduces the coefficient $1/2$ of \eqref{eq:pde} exactly: integrating $\partial_\tau U_n$ against $e^{-\tau}$ produces $\mathcal G_n-\phi_n$, while the linear operator passes through unchanged, so $\partial_\tau U_n=\tfrac12D_{\v}^2U_n$ becomes $\mathcal G_n-\tfrac12D_{\v}^2\mathcal G_n=\phi_n$. This is the normalization of~\cite[Remark~3.5]{bayraktar2020finite}, where the transform with discount rate $\lambda$ is matched against $\lambda u_\lambda-\tfrac12\max_{\v}D_{\v}^2u_\lambda=\phi_n$; our case is $\lambda=1$.
\end{remark}

\begin{remark}[What is not claimed]\label{rem:parnotclaimed}
Theorem~\ref{thm:finite} concerns optimality throughout the ordered sector. It does not rule out a fixed direction being optimal at a particular initial state, such as the origin, for a particular horizon. It does not identify the optimal switching regions, and it does not settle $C^2$ regularity of $U_n$ or of $u_n$.
\end{remark}

\appendix
\section{Bounds on the value at the origin}\label{app:origin}

This appendix proves the bounds on $u_n(0)$ quoted in Remark~\ref{rem:origin} and compares them. Throughout, $Z_1,\dots,Z_n$ are independent standard normal random variables, $X_n=\max_{i\leq n}Z_i$, and $\Phi$ is the standard normal distribution function. The upper bounds of this paper are
\begin{equation}\label{eq:originbounds}
 B_n=\frac{n-1}{n\kappa}\Bigl(1+\log\frac{n\kappa h_{n-1}(0)}{n-1}\Bigr),\qquad
 C_n=\sqrt{\frac{\pi\log n}8}.
\end{equation}
In the normalization of Remark~\ref{rem:origin}, the heat-based, exponential-weights and max-potential players of Kobzar, Kohn and Wang~\cite{kobzar2020newgeometric} give the upper bounds
\begin{equation}\label{eq:kkwbounds}
 H_n=\sqrt{\frac\pi8}\,\E X_n,\qquad
 W_n=\sqrt{\frac{\log n}2},\qquad
 M_n=\begin{cases}\sqrt{(n-1)/8},&n\text{ even},\\[2pt] \frac{n-1}n\sqrt{(n+1)/8},&n\text{ odd},\end{cases}
\end{equation}
and their heat-based adversary gives the lower bound $L_n=\tfrac{\sqrt\pi}4\sqrt{2c_n}\,\E X_n$, where $2c_n=n/(n-1)$ for even $n$ and $2c_n=(n+1)/n$ for odd $n$. The bound $M_n$ is the leading-order constant of the max-potential bound in~\cite[Section~7, Example~3]{kobzar2020newgeometric}. Some comparisons below rest on finitely many numerical inequalities, which are certified in interval arithmetic by \texttt{origin\_bounds.py}.

\begin{lemma}\label{lem:gaussmax}
For $n\geq2$ let $f_n=1-\Phi^n-(1-\Phi)^n$.
\begin{enumerate}
\item[(i)] $\E X_n=\int_0^\infty f_n(s)\,ds$, where $f_n$ is nonnegative and decreasing on $[0,\infty)$. Consequently $\E X_n\geq\ell_n:=10^{-3}\sum_{j=1}^{10^4}f_n(j/10^3)$. Moreover, $\E X_n$ is nondecreasing in $n$.
\item[(ii)] For every $n\geq9$, $\E X_n>\sqrt{\log n}$, that is, $C_n<H_n$.
\end{enumerate}
\end{lemma}

\begin{proof}
For (i), since $\Phi(-s)=1-\Phi(s)$,
\[
 \E X_n=\int_0^\infty\P(X_n>s)\,ds-\int_0^\infty\P(X_n<-s)\,ds
 =\int_0^\infty\bigl(1-\Phi(s)^n-(1-\Phi(s))^n\bigr)\,ds .
\]
On $[0,\infty)$ we have $1-\Phi\leq\Phi$, so $f_n'=n\Phi'\bigl((1-\Phi)^{n-1}-\Phi^{n-1}\bigr)\leq0$, and $f_n\geq0$ because $\Phi^n+(1-\Phi)^n\leq1$. A right Riemann sum of a nonnegative decreasing function is at most its integral, which gives $\E X_n\geq\ell_n$. Finally, $X_{n+1}\geq X_n$.

For (ii), suppose first that $n\geq10^{13}$, and let $b=\sqrt{(3/2)\log n}\geq1$. With $Q=1-\Phi$, the Mills ratio bound $Q(b)\geq b\Phi'(b)/(b^2+1)\geq\Phi'(b)/(2b)$ gives
\[
 nQ(b)\geq\frac{n^{1/4}}{2\sqrt{3\pi\log n}}>25,
\]
since the middle expression increases with $n$ and exceeds $25$ at $n=10^{13}$. Hence $\E(X_n)_+\geq b\,\P(X_n\geq b)\geq b\bigl(1-(1-Q(b))^n\bigr)>b(1-e^{-25})$, while $Q\leq\frac12$ on $[0,\infty)$ gives $\E(X_n)_-=\int_0^\infty Q(s)^n\,ds\leq2^{1-n}\int_0^\infty Q(s)\,ds=2^{1-n}/\sqrt{2\pi}$. Therefore
\[
 \E X_n\geq b\bigl(1-e^{-25}\bigr)-\frac{2^{1-n}}{\sqrt{2\pi}}>\sqrt{\log n}.
\]
For $9\leq n\leq10^{13}$ we use (i): if $\ell_m>\sqrt{\log N}$ for integers $m\leq N$, then $\E X_n\geq\E X_m\geq\ell_m>\sqrt{\log n}$ for every $m\leq n\leq N$. The script \texttt{origin\_bounds.py} encloses $\Phi$ by its Taylor series with the Lagrange remainder, evaluates $\ell_m$ in interval arithmetic, and finds eleven such blocks, starting from $m=9$, that cover $9\leq n\leq2.5\cdot10^{13}$.
\end{proof}

\begin{lemma}\label{lem:origin}
The following hold.
\begin{enumerate}
\item[(i)] For every $n\geq2$, $u_{n-1}(0)\leq u_n(0)\leq h_{n-1}(0)=(2n-2)!!/(4\kappa(2n-3)!!)$.
\item[(ii)] For every $n\geq4$, $u_n(0)\leq B_n<h_{n-1}(0)$.
\item[(iii)] For every $n\geq2$, $L_n\leq u_n(0)\leq C_n$, and $u_n(0)\sim C_n$ as $n\to\infty$. Moreover, $L_n>u_5(0)$ for every $n\geq6$.
\item[(iv)] The sequence $B_n/\sqrt{\log n}$ is increasing for $n\geq4$, and $B_n>C_n$ for every $n\geq14$.
\end{enumerate}
\end{lemma}

\begin{proof}
For (i), Lemma~\ref{lem:removal} with $j=n-1$ gives $u_{n-1}(0)\leq u_n(0)$, and \eqref{eq:upper} with $s=0$ and $z=0$ gives $u_n(0)\leq h_{n-1}(0)$, whose value is \eqref{eq:originprofile}.

For (ii), let $s\geq0$. The $n$ permutations of $(s,0,\dots,0)$ average to $(s/n)\one$, so convexity, permutation invariance and translation equivariance give $u_n(0)+s/n\leq u_n(s,0,\dots,0)$, while \eqref{eq:upper} gives $u_n(s,0,\dots,0)\leq s+h_{n-1}(0)e^{-\kappa s}$. Hence
\[
 u_n(0)\leq\Bigl(1-\frac1n\Bigr)s+h_{n-1}(0)e^{-\kappa s}\qquad\text{for every }s\geq0.
\]
The right side is strictly convex in $s$ and equals $h_{n-1}(0)$ at $s=0$. If $\kappa h_{n-1}(0)>(n-1)/n$, its minimum over $s\geq0$ is attained at $s_*=\kappa^{-1}\log\bigl(n\kappa h_{n-1}(0)/(n-1)\bigr)>0$, where its value is $B_n$, and therefore $B_n<h_{n-1}(0)$. By \eqref{eq:originprofile}, $\kappa h_{n-1}(0)=(2n-2)!!/(4(2n-3)!!)$ increases with $n$, since passing from $n$ to $n+1$ multiplies it by $2n/(2n-1)$, and it equals $4/5$, $32/35$ and $64/63$ for $n=4,5,6$. Thus $\kappa h_{n-1}(0)>(n-1)/n$ for every $n\geq4$.

For (iii), the lower bound is the heat-based bound of~\cite{kobzar2020newgeometric}, and the upper bound is proved in Remark~\ref{rem:finitegeometric}. Since $2c_n\to1$ and $\E X_n\sim\sqrt{2\log n}$, we have $L_n/C_n\to1$, so $u_n(0)\sim C_n$. Finally, since $2c_n>1$, Lemma~\ref{lem:gaussmax}(i) gives $L_n>\tfrac{\sqrt\pi}4\E X_n\geq\tfrac{\sqrt\pi}4\ell_8$ for $n\geq8$, and the certified values of $\ell_6$, $\ell_7$ and $\ell_8$ give
\[
 L_6-u_5(0)>0.0014,\qquad L_7-u_5(0)>0.026,\qquad \tfrac{\sqrt\pi}4\ell_8-u_5(0)>0.017.
\]

For (iv), note that $\kappa h_{n-1}(0)=a(n)$ with $a(x)=\tfrac{\sqrt\pi}4\Gamma(x)/\Gamma(x-\tfrac12)$, so for real $x\geq4$ we may write $\kappa B_x=(1-\tfrac1x)A(x)$ with $A(x)=1+\log\bigl(xa(x)/(x-1)\bigr)$. Log-convexity of $\Gamma$ gives $\Gamma(x)^2\leq\Gamma(x-\tfrac12)\Gamma(x+\tfrac12)=(x-\tfrac12)\Gamma(x-\tfrac12)^2$, so $a(x)<\sqrt{\pi x}/4$, and since $e\sqrt{\pi x}\leq4(x-1)$ for $x\geq4$, this gives $A(x)\leq\log x$. Moreover, $(\log\Gamma)''(t)=\sum_{k\geq0}(t+k)^{-2}$ is decreasing and exceeds $1/t$, so $(\log a)'(x)=\int_{x-1/2}^x(\log\Gamma)''(t)\,dt>1/(2x)$ and $A'(x)>(x-3)/(2x(x-1))$. Using these bounds and $x-1>2\log x$, we obtain for $x\geq4$
\[
 \frac{d}{dx}\log\frac{B_x}{\sqrt{\log x}}
 =\frac1{x(x-1)}+\frac{A'(x)}{A(x)}-\frac1{2x\log x}
 >\frac{\log x-1}{x(x-1)A(x)}>0 .
\]
Thus $B_n/\sqrt{\log n}$ is increasing for $n\geq4$, and it remains to check that $B_{14}>C_{14}$. By \eqref{eq:originbounds}, $\kappa B_{14}=\tfrac{13}{14}(1+\log T)$ with $T=14\kappa h_{13}(0)/13=2^{22}/2414425$, and $(\kappa C_{14})^2=\tfrac\pi4\log14$. Partial sums of the exponential series, with a geometric bound for the tail in the first case, show that $e^{69/125}<T$ and $e^{66/25}>14$. Together with $\pi<22/7$, this gives
\[
 (\kappa B_{14})^2-(\kappa C_{14})^2>\Bigl(\frac{13}{14}\Bigl(1+\frac{69}{125}\Bigr)\Bigr)^2-\frac{22}7\cdot\frac{66}{25}\cdot\frac14=\frac{1996}{765625}>0.
\]
These rational inequalities are checked exactly in \texttt{origin\_bounds.py}.
\end{proof}

\begin{remark}[Numerical comparison]\label{rem:origintable}
The table below lists the bounds for $4\leq n\leq14$, together with the exact values of~\cite{calder2026explicit} for $n\leq5$; the smallest upper bound in each row is in bold. The displayed values are rounded. The finite strict comparisons below are certified in interval arithmetic by \texttt{origin\_bounds.py}, using $H_n\geq\sqrt{\pi/8}\,\ell_n$ from Lemma~\ref{lem:gaussmax}(i); the equalities for $n=2,3$ follow from the exact formulas, and the comparisons for larger $n$ follow from Lemmas~\ref{lem:gaussmax} and~\ref{lem:origin} and the arguments below.
\begin{center}
\begin{tabular}{rcccccccc}
\toprule
 & & & \multicolumn{3}{c}{upper bounds, this paper} & \multicolumn{3}{c}{upper bounds of~\cite{kobzar2020newgeometric}}\\
\cmidrule(lr){4-6}\cmidrule(lr){7-9}
$n$ & $u_n(0)$ & $L_n$ & $h_{n-1}(0)$ & $B_n$ & $C_n$ & $H_n$ & $W_n$ & $M_n$\\
\midrule
4 & 0.5554 & 0.5267 & 0.5657 & \textbf{0.5646} & 0.7378 & 0.6451 & 0.8326 & 0.6124 \\
5 & 0.6134 & 0.5645 & 0.6465 & \textbf{0.6412} & 0.7950 & 0.7288 & 0.8971 & 0.6928 \\
6 & -- & 0.6151 & 0.7183 & \textbf{0.7060} & 0.8388 & 0.7941 & 0.9465 & 0.7906 \\
7 & -- & 0.6405 & 0.7836 & \textbf{0.7618} & 0.8742 & 0.8474 & 0.9864 & 0.8571 \\
8 & -- & 0.6744 & 0.8439 & \textbf{0.8108} & 0.9037 & 0.8921 & 1.0197 & 0.9354 \\
9 & -- & 0.6936 & 0.9002 & \textbf{0.8543} & 0.9289 & 0.9306 & 1.0481 & 0.9938 \\
10 & -- & 0.7187 & 0.9531 & \textbf{0.8935} & 0.9509 & 0.9643 & 1.0730 & 1.0607 \\
11 & -- & 0.7342 & 1.0033 & \textbf{0.9290} & 0.9704 & 0.9942 & 1.0950 & 1.1134 \\
12 & -- & 0.7540 & 1.0511 & \textbf{0.9615} & 0.9878 & 1.0210 & 1.1147 & 1.1726 \\
13 & -- & 0.7670 & 1.0968 & \textbf{0.9915} & 1.0036 & 1.0453 & 1.1325 & 1.2211 \\
14 & -- & 0.7833 & 1.1406 & 1.0192 & \textbf{1.0180} & 1.0674 & 1.1487 & 1.2748 \\
\bottomrule
\end{tabular}
\end{center}
For $n=2,3$ the exact values equal $h_{n-1}(0)$, since $\kappa u_2(0)=\kappa h_1(0)=1/2$ and $\kappa u_3(0)=\kappa h_2(0)=2/3$, and for $n=4,5$ they lie below every upper bound. The bound $B_n$ is smaller than $H_n$, $W_n$ and $M_n$ for $4\leq n\leq23$ and smaller than $C_n$ for $4\leq n\leq13$, so it is the smallest of these upper bounds for $6\leq n\leq13$. For $n\geq9$, the bound $C_n$ is smaller than $H_n$ by Lemma~\ref{lem:gaussmax}(ii), smaller than $W_n$ because $\pi<4$, and smaller than $M_n$ because $M_n^2\geq(n-2)/8>(\pi\log n)/8$. Together with $C_n<B_n<h_{n-1}(0)$ for $n\geq14$ from Lemma~\ref{lem:origin}, this makes $C_n$ the smallest of these upper bounds for every $n\geq14$.
\end{remark}

\bibliographystyle{abbrv}
\bibliography{main}

\end{document}